\documentclass[10pt]{article}

\usepackage{amssymb,latexsym,amsmath,enumerate,verbatim,amsfonts,amsthm}
\usepackage{color} 
\usepackage[active]{srcltx}

\newtheorem{theorem}{Theorem}
\newtheorem{corollary}{Corollary}

\newtheorem{remark}{Remark}
\newtheorem{example}{Example}

\def\cH{{\mathcal{H}}}

\def\R{{\rm I\!R}}
\def\M{{\mathcal{M}}}
\def\A{{\mathcal{A}}}
\def\B{{\mathcal{B}}}
\def\I{{\mathcal{I}}}
\def\cR{{\mathcal{R}}}
\def\T{{\mathcal{T}}}
\def\S{{\mathcal{S}}}
\def\X{{\mathcal{X}}}
\def\Y{{\mathcal{Y}}}
\def\Z{{\mathcal{Z}}}

\def\argmin{\mathop{\rm arg\,min}}

\def\Argmin{\mathop{\rm Arg\,min}}

\title{\sf The proximal-proximal gradient algorithm}

\author{Ting Kei Pong \thanks{Department of Computer Science,
University of British Columbia, Vancouver, BC, Canada V6T 1Z4. E-mail: {tkpong@cs.ubc.ca}.
This author is a PIMS postdoctoral fellow.}
}

\date{August 23, 2013} 

\begin{document}
\maketitle

\begin{abstract}
  We consider the problem of minimizing a convex objective which
  is the sum of a smooth part, with Lipschitz continuous gradient, and a nonsmooth part.
  Inspired by various applications, we focus on the case when the nonsmooth part is a composition of a proper closed convex function $P$
  and a nonzero affine map, with the proximal mappings of $\tau P$, $\tau > 0$, easy to compute.
  In this case, a direct application of the widely used proximal gradient algorithm
  does not necessarily lead to easy subproblems. In view of this, we propose a new algorithm, the proximal-proximal gradient algorithm,
  which admits easy subproblems. Our algorithm
  reduces to the proximal gradient algorithm if the affine map is just the identity map and the stepsizes are suitably chosen, and it
  is equivalent to applying a variant of the alternating minimization algorithm \cite{Tse91} to the dual problem.
  Moreover, it is closely related to inexact proximal gradient algorithms \cite{Patriksson98,SchRouBach11}.
  We show that the whole sequence generated from the algorithm
  converges to an optimal solution. We also establish an upper bound on iteration complexity.
  Our numerical experiments on the stochastic realization problem and the logistic fused lasso problem suggest that the algorithm
  performs reasonably well on large-scale instances.
\end{abstract}

\section{Introduction}

We consider the following optimization problem
\begin{equation}\label{FBS}
  \begin{array}{rl}
    v_{\rm opt} := \min\limits_z & F(z) := h(z) + P(\M z-b),
  \end{array}
\end{equation}
where $h$ is a convex differentiable function with a Lipschitz continuous gradient whose Lipschitz continuity modulus is
bounded by $L$,
$P$ is a proper closed convex function, $\M$ is a nonzero linear map between two finite dimensional Hilbert spaces $\Z$ and $\Y$, and $b$ is
a given element in $\Y$. We assume throughout that \eqref{FBS} has an optimal solution. We also assume the following constraint qualification
to rule out degenerate instances:
\begin{equation}\label{CQ}
  ({\rm Range}(\M) - b)\cap {\rm ri}({\rm dom}\,P)\neq \emptyset;
\end{equation}
see Section~\ref{sec1.1} for notations and definitions.
This condition is trivially satisfied if $P$ is continuous or $\M$ is surjective. Furthermore,
we assume that the proximal mapping of $P$, i.e., the quantity
\[
{\rm prox}_P(y) := \argmin_u \left\{P(u) + \frac{1}{2}\|u - y\|^2\right\}
\]
is easy to compute for any given $y$, so is that of $\tau P$ for any $\tau > 0$.\footnote{Here and throughout the paper,
by ``easy to compute", we mean an easy, and preferably a closed form solution, is available.}

Problem \eqref{FBS} arises from various applications, with $h$ usually the loss function, $P$ the regularization function
and $\M z - b$ modeling some priors known about the desired solution. One example arises from
the maximum a posteriori probability analysis approach in signal processing \cite{ElMiRu07}, where $b = 0$, $P$ is
the $\ell_1$ norm and $\M$ is the so-called analyzing operator,
with a higher value of $P(\M z)$ for a less probable signal. Another example arises from the system identification or system realization problems, where $P$ is the nuclear norm of a matrix (i.e., sum of singular values of a matrix), $b=0$
and $\M z$ is a block Hankel matrix of suitable size; see, for example, \cite{FPST12,LiV08,LiV09,LHV12,V12}.
Finally, we also mention the fused lasso logistic regression problem introduced recently in \cite{MaZhang13}
in which $h$ is the smooth logistic loss function, $b=0$, $P$ is the $\ell_1$ norm and $\M$ is a suitably defined linear map.
We refer the readers to \cite[Section~5]{MaZhang13} for more details about this particular problem. Notice that
in all the above applications, the proximal mapping of $P$ and those of $\tau P$, $\tau > 0$ are computable via a closed form formula;
see, for example, \cite[Section~9.2]{ToSuSu12}.

Since $h$ is differentiable with Lipschitz continuous gradient and $P$ is nonsmooth, it is natural to consider a proximal gradient algorithm \cite{FukuMine81}
for solving \eqref{FBS}, where one updates
\begin{equation}\label{FBSsubproblem}
  z^{t+1} = \argmin_z \left\{ \langle\nabla h(z^t),z\rangle + P(\M z - b) + \frac{L}{2}\|z - z^t\|^2\right\}.
\end{equation}
For such an algorithm, it is well-known that $F(z^t) - F(z^*) = O\left(\frac{1}{t}\right)$;
see, for example, \cite[Theorem~2.1.14]{Nes03}.
Furthermore, the whole sequence generated is convergent to an optimal solution of \eqref{FBS}; see, for example, \cite[Section~2]{Tse91}.
This simple algorithm, together with its accelerated versions using Nesterov
extrapolation techniques \cite{Nes83,Nes03,Nes05,Nes052}, has recently been studied extensively in the literature; see, for example, \cite{BeckTeboulle09,BeckerCandesGrant11,FPST12,LanLuMon09,MGC10,MoF10,PTJY09,ToYu10,Tse10}.
While the proximal gradient algorithm has nice convergence property, for
the algorithm to be efficient, the optimization problem in \eqref{FBSsubproblem} should be
easy to solve. This problem is the same as computing the proximal mapping of $\frac{1}{L}P(\M\cdot - b)$.
While the proximal mapping of $\frac{1}{L}P$ is easy to compute, the proximal mapping of the composition
$\frac{1}{L}P(\M\cdot - b)$, however, is in general not trivial to compute.

To get around this difficulty, one natural way is to ``decouple" the function $P$ and the affine map.
One possible way is to reformulate the problem as
\begin{equation}\label{specialprimal}
  \begin{array}{rl}
    \min\limits_{z,u} & h(z) + P(u)\\
    {\rm s.t.}& u = \M z - b,
  \end{array}
\end{equation}
and apply a suitable algorithm, for example, the alternating direction method of multipliers (ADMM);
see, for example, \cite{Eckstein94,EckB92,FGlowinski83,Gabay83,GaM76,GlM75}. This algorithm consists of the following updates
\begin{equation*}
\left\{
\begin{split}
&z^{t+1} \in \Argmin_z \left\{ h(z) + \langle{y^t},\M z\rangle + \frac{\beta}{2}\|u^t - \M z + b\|^2 \right\},\\
&u^{t+1} = \argmin_u \left\{ P(u) - \langle{y^t},u\rangle + \frac{\beta}{2}\|u - \M z^{t+1} + b\|^2 \right\},\\
&y^{t+1} = y^t - \gamma\beta(u^{t+1} -\M z^{t+1} + b),
\end{split}
\right.
\end{equation*}
where $\beta > 0$ and $\gamma\in (0,\frac{\sqrt{5}+1}{2})$, and we use capital $\Argmin$ to indicate that the minimizer
may not be unique.
While the second subproblem is just a computation of the proximal mapping of $\frac{1}{\beta}P$,
the first subproblem involving $h$ is in general not easy to solve.
Recent variants of this method have {\em proximal terms} added to the objectives for the $z$-update and $u$-update, in order to
reduce the subproblem to a computation of the proximal mappings of $\frac{1}{\beta}h$ and $\frac{1}{\beta}P$, respectively;
see, for example, \cite{YaZ09,ZBO09}. However, the proximal mapping of $\frac{1}{\beta}h$ can still be difficult to compute.
In the recent work \cite{MaZhang13}, a variant of the above ADMM was proposed to solve \eqref{FBS} so that each subproblem
has a simple closed form solution.
For this method, it was only shown that any cluster point of the average of the sequence $\{(z^t,u^t)\}$ generated minimizes a Lagrangian
function of \eqref{specialprimal}.

Another possible way to ``decouple" the function $P$ and the affine map is to reformulate the problem into a convex-concave optimization problem.
In this approach, one first rewrite \eqref{FBS} as a convex-concave optimization problem as follows:
\begin{equation}\label{saddlept}
  \min_z \max_y \{h(z) + \langle y, \M z - b\rangle - P^*(y)\},
\end{equation}
where $P^*$ is the convex conjugate of $P$.
One can then apply a suitable algorithm, for example, Tseng's modified forward backward splitting (MFBS) method proposed in \cite{Tse00}, which
consists of the following updates in each iteration (see \cite[Example~5]{Tse00})
\begin{equation}\label{extragradient}
\left\{
\begin{split}
& v^t = \argmin_y\left\{ P^*(y) + \langle b - \M z^t,y\rangle + \frac{L_\M}{2\sigma}\|y - y^t\|^2\right\},\ \ \ \ \ \
 u^t = z^t - \frac{\sigma}{L_\M}(\nabla h(z^t) +\M^*y^t),\\
& y^{t+1} = v^t - \frac{\sigma}{L_\M}(\M z^t - \M u^t),\ \
 z^{t+1} = u^t - \frac{\sigma}{L_\M}(\nabla h(u^t) +\M^* v^t - \nabla h(z^t) -\M^* y^t),
\end{split}
\right.
\end{equation}
where $\sigma\in (0,1)$ and $L_\M$ is an upper bound of the Lipschitz continuity modulus of the function
\[
G(z,y) = \begin{pmatrix}
  \nabla h(z) + \M^* y\\
  b - \M z
\end{pmatrix},
\]
and $\M^*$ is the adjoint linear map of $\M$.
The subproblem in \eqref{extragradient} can be shown to be equivalent to a computation of the proximal mapping of $\frac{L_\M}{\sigma}P$,
and it was further shown in \cite[Theorem~3.4(b)]{Tse00} that the sequence $\{(u^t,v^t)\}$ converges to a saddle point of \eqref{saddlept}.

In this paper, we propose a new algorithm, called the proximal-proximal gradient algorithm,
to solve \eqref{FBS}, which has nice convergence property and admits easy subproblems in each iteration like
the MFBS method mentioned above. Our algorithm is obtained by applying
a variant of the alternating minimization algorithm \cite{Tse91} (AMA) to the dual problem of \eqref{FBS}.
It reduces to the proximal gradient algorithm if $\M$ is the identity map, $b=0$ and the stepsizes are suitably chosen;
and is closely related to inexact proximal gradient algorithms \cite{Patriksson98,SchRouBach11}.
For this new algorithm, we prove that the whole sequence generated converges to an optimal solution of \eqref{FBS}
and establish an upper bound on iteration complexity. We also test our algorithm
on the stochastic realization problem (see, for example, \cite{LiV09}) against the primal ADMM2 in \cite{FPST12} and
the MFBS method in \cite{Tse00},
as well as on the fused lasso logistic regression problem (see \cite{MaZhang13}) against
the MFBS method in \cite{Tse00}.
Our computational results show that our algorithm works reasonably well on large-scale instances
and usually outperforms the MFBS method.

The rest of the paper is organized as follows. We introduce notations used throughout the paper in Section~\ref{sec1.1}
and motivate our algorithm in Section~\ref{sec2}. In Section~\ref{sec:AMA}, we propose a proximal alternating minimization algorithm
and establish its global convergence. In Section~\ref{sec4}, we present our proximal-proximal gradient algorithm and
discuss its various properties, including convergence and its relationship with inexact proximal gradient algorithms.
Numerical results are reported in Section~\ref{sec5}. Finally, we give concluding remarks in Section~\ref{sec6}.

\subsection{Notations}\label{sec1.1}

In this paper, scripted letters $\X$, $\Y$ and $\Z$ denote finite dimensional Hilbert spaces.
Abusing notations, we use $\langle\cdot,\cdot\rangle$
to denote the inner product and $\|\cdot\|$ to denote the norm induced from the inner product on each of the spaces.
Linear maps between finite dimensional Hilbert spaces
are denoted by scripted letters (other than $\X$, $\Y$ and $\Z$). For a linear map $\A$, $\A^*$ denotes the adjoint linear map with respect
to the inner product, ${\rm Range}(\A)$ denotes the image of $\A$ and
$\|\A\|$ is the induced operator norm of $\A$.
A linear self-map $\T:\Z\rightarrow \Z$ is called positive semidefinite (resp., positive definite) if
$\T = \T^*$ and $\langle z,\T z\rangle\ge 0$ (resp., $\langle z,\T z\rangle> 0$) for all nonzero $z\in \Z$. We let $\T \succeq 0$ (resp., $\T\succ 0$) denote $\T$ is a positive semidefinite
(resp., positive definite) linear map. For a positive semidefinite linear map $\T$, $\|\cdot\|_\T$ denotes its induced seminorm defined by
$\|z\|_\T := \sqrt{\langle z,\T z\rangle}$ for all $z\in \Z$.
We denote the identity map by $\I$.

For a closed convex function $f:\Z\rightarrow (-\infty,\infty]$, the domain of $f$ is a convex set and is defined by
${\rm dom}\,f :=\{z\in \Z:\; f(z) < \infty\}$.
A closed convex function $f$ is called proper if ${\rm dom}\,f\neq\emptyset$. The set of subdifferentials of $f$
at a point $z\in \Z$, denoted by $\partial f(z)$, is defined by
\[
\partial f(z) := \{v\in \Z:\; f(u) - f(z)\ge \langle v,u - z\rangle\ \ \forall u\in \Z\}.
\]

For a proper closed convex function $f:\Z\rightarrow (-\infty,\infty]$, the convex conjugate $f^*$ of $f$ is the proper closed convex function defined by
\[
f^*(z) := \sup_{u}\{\langle z,u\rangle - f(u)\}.
\]
It is well-known \cite[Theorem~12.2]{Roc70} that
\begin{equation}\label{biconjugate}
f^{**} = f.
\end{equation}
The proximal mapping of $f$ is defined by
\[
{\rm prox}_f(z) := \argmin_{u}\left\{f(u)+ \frac{1}{2}\|u - z\|^2\right\},
\]
where $\argmin$ denotes the (unique) minimizer. Recall from \cite[Theorem~31.5]{Roc70} that the minimizer of the above optimization problem always exists
and is unique, hence the proximal mapping is well-defined. Moreover, from the same theorem, we have the following relation:
\begin{equation}\label{proxmap}
{\rm prox}_f(z) + {\rm prox}_{f^*}(z) = z\ \ \ \forall z\in \Z.
\end{equation}
Finally, we use ${\rm ri}(C)$ to denote the relative interior of a convex set $C$.

\section{Motivations}\label{sec2}

Recall that the main difficulty for the efficient implementation of the proximal gradient algorithm \eqref{FBSsubproblem} lies in the fact that
the proximal mapping of $\tau P(\M\cdot - b)$, $\tau > 0$, is not necessarily easy to compute. In this section, we will try to look
at the proximal gradient algorithm \eqref{FBSsubproblem} from a dual perspective. In particular, we revisit the well-known fact that
the proximal gradient algorithm is equivalent to applying the AMA to the dual of \eqref{FBS}, a result established in \cite[Section~2]{Tse91}.

To describe the equivalence, we note from \eqref{CQ} and \cite[Theorem~31.2]{Roc70} that
\begin{align*}
  v_{\rm opt}  = \min_{z}h(z) + P(\M z - b)
     = \max_y -h^*(-\M^*y) - P^*(y) - \langle b,y\rangle,
\end{align*}
where $h^*$ and $P^*$ are the convex conjugate of $h$ and $P$ respectively.
Thus, the dual of \eqref{FBS} is equivalent to the following problem:
\begin{equation}\label{fencheldual}
  \begin{array}{rl}
    \min\limits_{x,y}& h^*(x) + P^*(y) + \langle b,y\rangle\\
    {\rm s.t.}& x + \M^*y = 0.
  \end{array}
\end{equation}

Since $\nabla h$ is Lipschitz continuous with modulus bounded by $L$, the function
$h^*$ is strongly convex with strong convexity modulus at least $\frac{1}{L}$.
Hence, one can apply the AMA in \cite[Section~4]{Tse91} to solve the dual problem \eqref{fencheldual}:
\begin{equation}\label{AMAsubproblem}
\left\{
\begin{split}
 & x^{t+1} = \argmin_x \left\{ h^*(x) - \langle {z^t},x\rangle\right\},\\
 & y^{t+1} \in \Argmin_y \left\{ P^*(y) + \langle b,y\rangle - \langle {z^t}, \M^*y\rangle + \frac{\beta}{2}\|x^{t+1} + \M^*y\|^2\right\},\\
 & z^{t+1} = z^t - \beta (x^{t+1} + \M^*y^{t+1}).
\end{split}
\right.
\end{equation}
This algorithm is known to converge if $\beta \in (0,\frac{2}{L})$, assuming that the iterates are well-defined;
see \cite[Proposition~3]{Tse91} for the precise assumptions. Moreover, with $\beta = \frac{1}{L}$,
it can be shown (see \cite[Section~2]{Tse91}) that this algorithm is equivalent to the proximal gradient algorithm \eqref{FBSsubproblem},
in the sense that the $z$-iterates generated are the same.
To be precise, we include the simple arguments below for completeness.

Indeed, from the first subproblem of \eqref{AMAsubproblem} we see that
\begin{equation}\label{xupdateAMA}
x^{t+1} = \argmin_x \left\{ h^*(x) - \langle {z^t},x\rangle\right\} \Leftrightarrow z^t\in \partial h^*(x^{t+1}) \Leftrightarrow x^{t+1} = \nabla h(z^t),
\end{equation}
where the last equivalence follows from \cite[Theorem~23.5]{Roc70}.
On the other hand, from the second and third subproblem of \eqref{AMAsubproblem}, we see that
\begin{equation}\label{yupdateAMA}
\begin{split}
    &  y^{t+1} \in \Argmin_y \left\{ P^*(y) + \langle b,y\rangle - \langle {z^t}, \M^*y\rangle + \frac{\beta}{2}\|x^{t+1} + \M^*y\|^2\right\}\\
    & \Leftrightarrow \M z^{t+1} - b \in \partial P^*(y^{t+1}) \Leftrightarrow y^{t+1}\in \partial P(\M z^{t+1} - b) \\
    & \Rightarrow
    \M^* y^{t+1}\in \M^*\partial P(\M z^{t+1} - b) \Rightarrow \M^* y^{t+1} \in \partial P(\M \cdot - b)(z^{t+1}),
  \end{split}
\end{equation}
where the second equivalence follows from \cite[Theorem~23.5]{Roc70} and the last implication follows from \cite[Theorem~23.9]{Roc70}.
From these it is easy to verify that $\{z^t\}$ generated from the AMA satisfies \eqref{FBSsubproblem}. Conversely, if $\{z^t\}$ is generated according to \eqref{FBSsubproblem}, then from the first-order optimality condition, \eqref{CQ} and the subdifferential calculus rules \cite[Theorem~23.8]{Roc70}, \cite[Theorem~23.9]{Roc70}, we have
\begin{equation*}
  0 \in \nabla h(z^t) + \M^* \partial P(\M z^{t+1} - b) + L (z^{t+1} - z^t).
\end{equation*}
Define $x^{t+1} = \nabla h(z^t)$ and let $y^{t+1}\in \partial P(\M z^{t+1} - b)$ so that
$ 0 = \nabla h(z^t) + \M^* y^{t+1} + L (z^{t+1} - z^t)$. Then it is easy to see that
\begin{equation*}
  z^{t+1} = z^t - \frac{1}{L}(x^{t+1} + \M^* y^{t+1}).
\end{equation*}
Using these relations, $\beta = \frac{1}{L}$ and the
equivalences in \eqref{xupdateAMA} and \eqref{yupdateAMA}, we observe that the sequence $\{(x^t,y^t,z^t)\}$ are the same iterates as generated from the AMA.

Since the AMA is in some sense equivalent to the proximal gradient algorithm, one should expect that the second subproblem involving $P^*$ is no easier than \eqref{FBSsubproblem}. However, we do gain some insights from this dual perspective.
In the next section,
borrowing the idea of adding a proximal term which was recently discussed in \cite{ZBO09} in the settings of the ADMM,
we develop a proximal alternating minimization algorithm which modifies the $y$-update in \eqref{AMAsubproblem}.
The resulting algorithm is then applied to \eqref{FBS} in Section~\ref{sec4} to give our proximal-proximal gradient algorithm.

\section{A proximal alternating minimization algorithm}\label{sec:AMA}

In this section, we consider a variant of the alternating minimization algorithm (AMA) proposed in \cite[Section~4]{Tse91}.
This method aims at solving optimization problems in the following form:
\begin{align}\label{genprimal}
\begin{array}{rl} \min\limits_{x,y} & f(x)+g(y) \\ {\rm s.t.} &  \A x+\B y=c,
\end{array}
\end{align}
where $\A: \X\rightarrow \Z$ and $\B:\Y\rightarrow \Z$ are linear maps on the finite dimensional Hilbert spaces $\X$, $\Y$ and $\Z$, $c\in \Z$,
$f: \X \rightarrow (-\infty,+\infty]$ and $g:\Y \rightarrow (-\infty,+\infty]$ are proper closed convex functions,
with linear maps $\Sigma_f\succ 0$ and $\Sigma_g\succeq 0$ such that for any $u_1\in \partial f(x_1)$
and $u_2\in \partial f(x_2)$, we have
\begin{equation}\label{strongconvexf}
  \langle u_1-u_2,x_1-x_2\rangle \ge \|x_1-x_2\|_{\Sigma_f}^2;
\end{equation}
while for any $v_1\in \partial g(y_1)$
and $v_2\in \partial g(y_2)$, we have
\begin{equation}\label{convexg}
  \langle v_1-v_2,y_1-y_2\rangle \ge \|y_1-y_2\|_{\Sigma_g}^2,
\end{equation}
where $\|\cdot\|_\T$ is the semi-norm induced by a positive semidefinite linear map $\T$.
We add a proximal term in the second subproblem of the AMA and obtain the following so-called
proximal AMA for solving the above problem. It reduces to the original AMA if $\T = 0$ and $\gamma = 1$.

\vspace{.1in}
\fbox{\parbox{5.6 in}{
\begin{description}
\item {\bf Proximal AMA for \eqref{genprimal}}

\item[Step 0.] Input $(y^0, z^0) \in \Y\times \Z$.

\item[Step 1.] Set
\begin{equation}\label{scheme}
\left\{
\begin{split}
&x^{t+1}=\argmin_{x}\left\{f(x)-\langle z^t, \A x\rangle\right\}, \\
&y^{t+1}=\argmin_{y} \left\{g(y)-\langle z^t, \B y\rangle + \frac{\beta}{2} \|\A x^{t+1}+\B y-c\|^2+\frac{1}{2}\|y-y^t\|_\T^2\right\},\\
&z^{t+1}=z^t-\gamma\beta  (\A x^{t+1}+\B y^{t+1}-c),
\end{split}
\right.
\end{equation}
where $\beta >0$, $\gamma > 0$ and $\T$ is a positive semidefinite linear map.

\item[Step 2.] If a termination criterion is not met, go to Step 1.
\end{description}
}}
\vspace{.1in}

When we apply the proximal AMA, we should expect $\min_xf(x) + \langle a,x\rangle$ to have an easy/closed form solution
for any given $a$ so that the first subproblem in \eqref{scheme} is easy to solve.
The motivation for introducing the proximal term in the second subproblem is to simplify the subproblem for the $y$-update by reducing it to a computation
of the proximal mapping of $\frac{1}{\beta}g$, as discussed recently in \cite{ZBO09} in the settings of ADMM. This would be particularly useful
when the proximal mapping of $\frac{1}{\beta}g$ is easy to compute. Moreover, as discussed in \cite[Appendix~B]{FPST12},
$\T$ should be chosen as ``small" as possible in practice, and thus preferably not positive definite.

In the next theorem, we establish convergence of the above proximal AMA, under suitable assumptions.
The tools used in the proof are quite standard and is largely based on \cite{FGlowinski83} for the convergence of the alternating direction
method of multipliers (ADMM), a closely related algorithm.
See also \cite[Section~6]{Gabay83}, \cite[Appendix~A]{Tse91}, \cite{XuW2011} and \cite[Theorem~B.1]{FPST12}.
We also note that the convergence of AMA without the proximal term (i.e., $\T = 0$), $\gamma=1$ and with a varying step size $\beta$
was established in \cite[Proposition~3]{Tse91}, under slightly different assumptions from below.

Before stating the theorem, we make the following assumption:
\begin{enumerate}[{\bf {A}1.}]
  \item There exist $(\bar{x},\bar{y}) \in \X\times \Y$ and $\bar{z} \in \Z$ such that
\begin{equation}\label{genopt}
\A^*\bar{z} \in \partial f(\bar{x}), \quad \B^*\bar{z} \in \partial
g(\bar{y}), \quad \A\bar{x}+\B\bar{y}-c=0.
\end{equation}
\end{enumerate}
Note that for any $(\bar{x},\bar{y}) \in \X\times \Y$ and $\bar{z} \in \Z$ satisfying \eqref{genopt}, $(\bar{x},\bar{y})$
is an optimal solution to \eqref{genprimal} and $\bar z$ is an optimal solution to the dual. Moreover, under a suitable constraint qualification, optimal solutions
to the dual problem of \eqref{genprimal} exist (see \cite[Corollary~28.2.2]{Roc70}), with \eqref{genopt} being satisfied by any optimal solution $(\bar{x},\bar{y})$ to problem~\eqref{genprimal} (see \cite[Corollary~28.3.1]{Roc70}) and any optimal solution $\bar{z}$ to the dual problem of \eqref{genprimal} (see \cite[Corollary~28.4.1]{Roc70}).\footnote{We note here that it is very common in the literature
to use a constraint qualification in place of {\bf A1} as the standing assumption for a convergence result. However,
since we will soon apply this proximal AMA to the dual problem \eqref{fencheldual} of \eqref{FBS}, we chose to use {\bf A1}, which is a condition
readily shown to hold for \eqref{fencheldual} under our assumptions on \eqref{FBS}; see Theorem~\ref{them_ppg}.}


\begin{theorem}\label{them_ama}
Suppose that {\bf A1} holds and let $\{(x^t,y^t,z^t)\}$ be generated from the proximal AMA. Suppose
further that $\beta>0$, $\T\succeq 0$ and $\gamma > 0$ are chosen so that
\begin{enumerate}[{\rm (i)}]
  \item $\Sigma_g+\T+\beta \B^*\B\succ 0$;
  \item for some $\mu > 0$, $2\Sigma_f - (\beta + \mu) \A^*\A\succ 0$ and $\gamma < 1 + \frac{\min\{\beta,\mu\}}{2\beta}$.
\end{enumerate}
Then $\{(x^t,y^t,z^t)\}$
is convergent and the limit satisfies \eqref{genopt}. In particular,
$\{(x^t, y^t)\}$ converges to an optimal solution to \eqref{genprimal} and $\{z^t\}$ converges
to an optimal solution to the dual problem of \eqref{genprimal}.
\end{theorem}

\begin{proof}
We first observe from $\Sigma_f \succ 0$ and assumption {(i)} of the theorem that the iterates are all well-defined.
Furthermore, the iterates of the proximal AMA can be equivalently rewritten as follows:
\begin{equation}\label{inclusion}
\left\{
\begin{split}
& 0 \in \partial
f(x^{t+1})-\A^*z^t, \\
& 0 \in \partial
g(y^{t+1})-\B^*\left(z^t-\beta(\A x^{t+1}+\B y^{t+1}-c)\right)+\T(y^{t+1}-y^t),\\
& z^{t+1}=z^t-\gamma\beta(\A x^{t+1}+\B y^{t+1}-c).
\end{split}
\right.
\end{equation}
For notational simplicity, we fix any $(\bar x,\bar y, \bar z)$ satisfying \eqref{genopt} and write $w^t_e:=w^t-\bar{w}$, where $w$ represents $x$, $y$ and $z$ respectively.
From \eqref{inclusion}, \eqref{genopt}, \eqref{strongconvexf} and \eqref{convexg}, we have
\begin{equation}\label{forcite}
\langle \A^*z_e^t,
x_e^{t+1}\rangle \geq \|x_e^{t+1}\|_{\Sigma_f}^2,\ \
\langle \B^*\left(z_e^t-\beta(\A x_e^{t+1}+\B y_e^{t+1})\right)-\T(y^{t+1}-y^t),
y^{t+1}_e\rangle \geq \|y_e^{t+1}\|_{\Sigma_g}^2.
\end{equation}
Summing the above two inequalities and rearranging terms, we obtain
\begin{align*}
\langle z_e^t,
\A x_e^{t+1} + \B y^{t+1}_e\rangle
-\beta\langle \A x_e^{t+1}+\B y_e^{t+1},
\B y^{t+1}_e\rangle-\langle \T(y^{t+1}-y^t),
y^{t+1}_e\rangle \geq \|x_e^{t+1}\|_{\Sigma_f}^2 + \|y_e^{t+1}\|_{\Sigma_g}^2.
\end{align*}
Using the relation $\A x^{t+1}_e+\B y^{t+1}_e=(\gamma\beta)^{-1}(z^t-z^{t+1})$ to the first term on the left, we see that
\[
\frac{1}{\gamma\beta}\langle z_e^t,z^t-z^{t+1}\rangle
-\beta\langle \A x_e^{t+1}+\B y_e^{t+1},
\B y^{t+1}_e\rangle-\langle y^{t+1}-y^t,
y^{t+1}_e\rangle_\T \geq \|x_e^{t+1}\|_{\Sigma_f}^2 + \|y_e^{t+1}\|_{\Sigma_g}^2.
\]
Next, applying the elementary relations $\langle u,v \rangle = \frac{1}{2}(\|u\|^2+\|v\|^2-\|u-v\|^2)$ to the first two terms on the left and $\langle u,v \rangle_\T = \frac{1}{2}(\|u\|^2_\T +\|v\|^2_\T -\|u-v\|^2_\T )$
to the third term on the left, multiplying both sides by a factor of $2$
and rearranging terms, we arrive at
\begin{align}\label{usefulbd}
\begin{split}
&\left(\frac{1}{\gamma\beta}\|z_e^t\|^2 + \|y_e^t\|_\T^2\right) - \left(\frac{1}{\gamma\beta}\|z_e^{t+1}\|^2 + \|y_e^{t+1}\|_\T^2\right)\\
\ge &\  2\|x_e^{t+1}\|_{\Sigma_f}^2 + 2\|y_e^{t+1}\|_{\Sigma_g}^2 + \|y^{t+1}-y^t\|_\T^2 + \beta \|\B y_e^{t+1}\|^2 - \beta\|\A x_e^{t+1}\|^2 + \left(\frac{1}{\gamma^2\beta} - \frac{1}{\gamma\beta}\right)\|z^{t+1} - z^t\|^2\\
= &\  \|x_e^{t+1}\|_{2\Sigma_f-\beta \A^*\A}^2 + 2\|y_e^{t+1}\|_{\Sigma_g}^2 + \|y^{t+1}-y^t\|_\T^2 + \beta \|\B y_e^{t+1}\|^2 + \left(\frac{1}{\gamma^2\beta} - \frac{1}{\gamma\beta}\right)\|z^{t+1} - z^t\|^2\\
= &\  \|x_e^{t+1}\|_{\S}^2 + 2\|y_e^{t+1}\|_{\Sigma_g}^2 + \|y^{t+1}-y^t\|_\T^2 + \mu \|\A x_e^{t+1}\|^2 + \beta \|\B y_e^{t+1}\|^2 + \left(\frac{1}{\gamma^2\beta} - \frac{1}{\gamma\beta}\right)\|z^{t+1} - z^t\|^2\\
\ge &\ \|x_e^{t+1}\|_{\S}^2 + 2\|y_e^{t+1}\|_{\Sigma_g}^2 + \|y^{t+1}-y^t\|_\T^2 + (\beta - \alpha)\|\B y_e^{t+1}\|^2 + \left(\frac{1}{\gamma^2\beta} - \frac{1}{\gamma\beta}\right)\|z^{t+1} - z^t\|^2\\
    &\ \ + \frac{\min\{\alpha,\mu\}}{2}\|\A x^{t+1}_e + \B y_e^{t+1}\|^2\\
= &\ \|x_e^{t+1}\|_{\S}^2 + 2\|y_e^{t+1}\|_{\Sigma_g}^2 + \|y^{t+1}-y^t\|_\T^2 + (\beta - \alpha)\|\B y_e^{t+1}\|^2 + \frac{1}{\gamma^2\beta}\left(\frac{\min\{\alpha,\mu\}}{2\beta} + 1- \gamma\right)\|z^{t+1} - z^t\|^2,
\end{split}
\end{align}
where $\S = 2\Sigma_f-(\beta + \mu) \A^*\A$, $\alpha > 0$
and we used the relation $\A x^{t+1}_e+\B y^{t+1}_e=(\gamma\beta)^{-1}(z^t-z^{t+1})$ in the first inequality
and the last equality. We further choose $\alpha > 0$ in \eqref{usefulbd} so that $\alpha < \beta$ and
\[
\gamma < \frac{\min\{\alpha,\mu\}}{2\beta} + 1;
\]
this is possible due to assumption {(ii)} of the theorem. From this choice and \eqref{usefulbd}, it follows that
the sequence $\{\frac{1}{\gamma\beta}\|z_e^t\|^2 + \|y_e^t\|_\T^2\}$ is monotonically nonincreasing (and so, bounded), and
\begin{equation}\label{padmineq1}
\lim_{t\rightarrow \infty} \|x_e^{t+1}\|_{\S}^2 + \|y_e^{t+1}\|_{\Sigma_g}^2 + \|y^{t+1}-y^t\|_\T^2 + \|\B y_e^{t+1}\|^2 + \|z^{t+1} - z^t\|^2 =0.
\end{equation}
These together with the positive definiteness of $\Sigma_g+\T+\beta \B^*\B$ and $\S$
establish the boundedness of $\{(x^t,y^t,z^t)\}$. It then follows immediately that
there exists a subsequence $\{(x^{t_i},y^{t_i},z^{t_i})\}$ that converges to a limit point
$(x^\diamond, y^\diamond, z^\diamond)$. We next show that $(x^\diamond, y^\diamond, z^\diamond)$ satisfies \eqref{genopt}.

To see this, we first observe from \eqref{padmineq1} and the positive definiteness of $\S$  that
\begin{equation}\label{AST}
\begin{split}
\lim_{t\rightarrow \infty}\|x^{t+1}_e\| = 0, \ \lim_{t\rightarrow \infty} \|y^t_e\|_{\Sigma_g} + \|\B y^t_e\|= 0, \lim_{t\rightarrow \infty}\|y^{t+1}-y^t\|_{\T}=\lim_{t\rightarrow \infty}\|z^{t+1}-z^t\| =0.
\end{split}
\end{equation}
Taking limits on both sides of \eqref{inclusion} along the subsequence $\{(x^{t_i-1},y^{t_i-1},z^{t_i-1})\}$, using
the third relation in \eqref{AST} and the closedness of the graphs of $\partial f$ and $\partial g$ \cite[Page~80]{BoLe06}, we conclude that
\[
\A^*z^\diamond \in \partial f(x^\diamond), \quad \B^*z^\diamond \in \partial g(y^\diamond), \quad \A x^\diamond+\B y^\diamond-c=0,
\]
showing that $(x^\diamond, y^\diamond, z^\diamond)$ satisfies \eqref{genopt}.

To complete the proof, now it remains to show that $(x^\diamond,y^\diamond,z^\diamond)$ is the unique limit of $\{(x^t,y^t,z^t)\}$. Since $(x^\diamond,y^\diamond,z^\diamond)$ satisfies \eqref{genopt}, we could replace $(\bar{x},\bar{y},\bar{z})$ with
$(x^\diamond,y^\diamond,z^\diamond)$ in the above arguments, starting from \eqref{forcite}.
First, the convergence
of $\{x^t\}$ follows immediately from \eqref{AST}.
Next, notice that the subsequence $\{\frac{1}{\gamma\beta}\|z_e^{t_i}\|^2 + \|y_e^{t_i}\|_\T^2\}$ converges to 0 as $i\rightarrow \infty$.
Since this sequence is also non-increasing, we must have
\begin{equation}\label{limit2}
\lim_{t\rightarrow \infty}
\frac{1}{\gamma\beta}\|z_e^t\|^2 + \|y_e^t\|_\T^2=0.
\end{equation}
From this, we see immediately that $\lim_{t\rightarrow \infty}z^t=z^\diamond$. Finally, using \eqref{limit2}, \eqref{AST} and the assumption
that $\Sigma_g +\T+\beta \B^*\B\succ 0$, we obtain further that $\lim_{t\rightarrow \infty}y^t=y^\diamond$. This completes the proof.
\end{proof}

Next, we explore the iteration complexity of the above algorithm. Obviously, one can characterize optimality
of a point $(\hat x,\hat y)$ by
\begin{equation*}
  f(\hat x) + g(\hat y) = f(\bar x) + g(\bar y) \ \ {\rm and}\ \ \|\A \hat x + \B \hat y - c\| = 0,
\end{equation*}
where $(\bar x,\bar y)$ is a solution to \eqref{genprimal}.
In the next theorem, we study the iteration complexity of the proximal AMA by establishing upper bound on $\|\A x + \B y - c\|$ along
a suitable sequence, and upper and lower bounds on the difference between the function value and
the optimal value along the same sequence. Our proof technique is similar to \cite[Lemma~2]{WangBan13} which established bounds similar to \eqref{fvalbd}
and \eqref{constraintbd} for their algorithm.

\begin{theorem}\label{them_complexity}
  Suppose that {\bf A1} holds and let $\{(x^t,y^t,z^t)\}$ be generated from the proximal AMA. Suppose
further that $\beta>0$, $\T\succeq 0$ and $\gamma > 0$ are chosen so that
\begin{enumerate}[{\rm (i)}]
  \item $\Sigma_g+\T+\beta \B^*\B\succ 0$;
  \item for some positive $\mu$, $\delta$ and $\sigma$,
  $2\Sigma_f - (\beta + \mu) \A^*\A\succeq \delta I$ and $\gamma + \sigma \le 1 + \frac{\min\{\beta,\mu\}}{2\beta}$.
\end{enumerate}
  Define
  \begin{equation*}
    (\bar x^N,\bar y^N) = \frac{1}{N}\sum_{t=1}^N (x^t,y^t).
  \end{equation*}
  Then it holds that
  \begin{equation}\label{fvalbd}
  \begin{split}
    &-\frac{\|\bar z\|}{\sqrt{N\sigma\beta}}
    \sqrt{\frac{1}{\gamma\beta}\|z^0 - \bar z\|^2 + \|y^0 - \bar y\|_\T^2} \le f(\bar x^N) + g(\bar y^N) - f(\bar x) - g(\bar y) \\
    &\le
    \frac{1}{2N}\left(\frac{1}{\gamma\beta}\|z^0\|^2 + \|y^0 - \bar y\|_\T^2\right) + \frac{1}{N}\left(\frac{\beta\|\A\|^2}{2\delta}+\frac{\max\{\gamma-1,0\}}{2\sigma}\right)\left(\frac{1}{\gamma\beta}\|z^0 - \bar z\|^2 + \|y^0 - \bar y\|_\T^2\right).
  \end{split}
  \end{equation}
  and
  \begin{equation}\label{constraintbd}
    \|\A \bar x^N + \B \bar y^N - c\| \le \frac{1}{\sqrt{N\sigma\beta}}
    \sqrt{\frac{1}{\gamma\beta}\|z^0 - \bar z\|^2 + \|y^0 - \bar y\|_\T^2},
  \end{equation}
  where $(\bar x,\bar y,\bar z)$ is the limit of $\{(x^t,y^t,z^t)\}$ guaranteed by Theorem~\ref{them_ama}.
\end{theorem}
\begin{proof}
  As before, let $(\bar x,\bar y,\bar z)$ be the limit of $\{(x^t,y^t,z^t)\}$,
  and write $w^t_e:=w^t-\bar w$, where $w$ represents $x$, $y$ and $z$ respectively. Then it holds that
  \begin{equation}\label{bdineq1}
  \begin{split}
    &\ \langle \A^*(z^{t+1} - z^t), x^{t+1}_e\rangle + \langle \T(y^{t+1} - y^t) + (1-\gamma)\beta\B^*(\A x_e^{t+1} + \B y_e^{t+1}), y^{t+1}_e\rangle \\
    = &\ \frac{1}{\gamma\beta}\langle z^{t+1} - z^t, \gamma\beta\A x^{t+1}_e\rangle + \langle y^{t+1} - y^t, y^{t+1}_e\rangle_\T
    + (1-\gamma)\beta\langle \A x_e^{t+1} + \B y_e^{t+1}, \B y^{t+1}_e\rangle\\
    = &\ \frac{1}{2\gamma \beta}\|z^{t+1} - z^t + \gamma \beta\A x^{t+1}_e\|^2 - \frac{1}{2\gamma\beta}\|z^{t+1} - z^t\|^2
    - \frac{\gamma\beta}{2}\|\A x^{t+1}_e\|^2\\
      &\ + \frac{1}{2}\|y^{t+1} - y^t\|_\T^2 + \frac{1}{2}\|y^{t+1}_e\|_\T^2 - \frac{1}{2}\|y^{t}_e\|_\T^2
      + \frac{(1-\gamma)\beta}{2}\left(\|\A x_e^{t+1} + \B y_e^{t+1}\|^2 + \|\B y_e^{t+1}\|^2 - \|\A x_e^{t+1}\|^2\right)\\
    = & \ \frac{\beta}{2}\|\B y^{t+1}_e\|^2 - \frac{1}{2\gamma\beta}\|z^{t+1} - z^t\|^2
    - \frac{\beta}{2}\|\A x^{t+1}_e\|^2\\
      &\ + \frac{1}{2}\|y^{t+1} - y^t\|_\T^2 + \frac{1}{2}\|y^{t+1}_e\|_\T^2 - \frac{1}{2}\|y^{t}_e\|_\T^2
      +\frac{(1-\gamma)\beta}{2}\|\A x_e^{t+1} + \B y_e^{t+1}\|^2\\
    = & \ \frac{\beta}{2}\|\B y^{t+1}_e\|^2 + \frac{1}{2\gamma\beta}\left(\frac{1 -\gamma}{\gamma} - 1\right)\|z^{t+1} - z^t\|^2
    - \frac{\beta}{2}\|\A x^{t+1}_e\|^2 + \frac{1}{2}\|y^{t+1} - y^t\|_\T^2 + \frac{1}{2}\|y^{t+1}_e\|_\T^2 - \frac{1}{2}\|y^{t}_e\|_\T^2\\
  \end{split}
  \end{equation}
  where we made use of the elementary relations $\langle u,v \rangle = \frac{1}{2}(\|u\|^2+\|v\|^2-\|u-v\|^2) = \frac{1}{2}(\|u + v\|^2-\|u\|^2-\|v\|^2)$
  and $\langle u,v \rangle_\T = \frac{1}{2}(\|u\|_\T^2+\|v\|_\T^2-\|u-v\|_\T^2)$
  in the second equality, the fact that $-\A \bar x  = \B \bar y - c$ in the third equality and the definition of $z^{t+1}$ in the last equality.
  Moreover, note from \eqref{inclusion} that the iterates of the proximal AMA satisfy the following relations:
  \begin{equation}\label{thm2optcon}
  \left\{
    \begin{split}
      &-\A^*(z^{t+1} - z^t) \in \partial f(x^{t+1}) - \A^* z^{t+1},\\
      &-(1-\gamma)\beta\B^*(\A x_e^{t+1} + \B y_e^{t+1})-\T(y^{t+1} - y^t)\in \partial g(y^{t+1}) - \B^* z^{t+1},\\
      &-\frac{1}{\gamma\beta}(z^{t+1} - z^t) = \A x^{t+1} + \B y^{t+1} - c.
    \end{split}
  \right.
  \end{equation}
  Using this, \eqref{bdineq1}, and the convexity of the Lagrangian function
  $\ell(x,y,z):= f(x) + g(y)-\langle z, \A x + \B y - c\rangle$ in the variables $x$ and $y$, we obtain further that
  \begin{equation*}
    \begin{split}
     &\ \frac{1}{2\gamma\beta}\left(\frac{1 -\gamma}{\gamma} - 1\right)\|z^{t+1} - z^t\|^2
    - \frac{\beta}{2}\|\A x^{t+1}_e\|^2 + \frac{1}{2}\|y^{t+1}_e\|_\T^2 - \frac{1}{2}\|y^{t}_e\|_\T^2\\
     \le &\ \langle -\A^*(z^{t+1} - z^t), \bar x - x^{t+1}\rangle + \langle - (1-\gamma)\beta\B^*(\A x_e^{t+1} + \B y_e^{t+1})-\T(y^{t+1} - y^t), \bar y - y^{t+1}\rangle\\
     \le &\ \ell(\bar x,\bar y,z^{t+1}) - \ell(x^{t+1},y^{t+1},z^{t+1}) = f(\bar x) + g(\bar y) - f(x^{t+1}) - g(y^{t+1}) + \langle z^{t+1},\A x^{t+1}_e + \B y^{t+1}_e\rangle\\
     = &  f(\bar x) + g(\bar y) - f(x^{t+1}) - g(y^{t+1}) +
     \frac{1}{2\gamma\beta}\|z^t\|^2 - \frac{1}{2\gamma\beta}\|z^{t+1}\|^2 - \frac{1}{2\gamma\beta}\|z^{t+1} - z^t\|^2,
    \end{split}
  \end{equation*}
  where we made use of the third relation in \eqref{thm2optcon} and the relation $\langle u,v \rangle = \frac{1}{2}(\|u + v\|^2-\|u\|^2-\|v\|^2)$
  in the last equality.

  Rearranging terms in the above inequality and summing from $t=0$ to $N-1$, we obtain that
  \begin{equation}\label{bdineq2}
    \sum_{t=0}^{N-1}\left(f(x^{t+1}) + g(y^{t+1}) - f(\bar x) - g(\bar y)\right) \le \frac{1}{2}\left(\frac{1}{\gamma\beta}\|z^0\|^2 + \|y_e^0\|_\T^2\right) + \frac{\beta}{2}\sum_{t=0}^{N-1}\|\A x^{t+1}_e\|^2 + \frac{\gamma-1}{2\gamma^2\beta}\sum_{t=0}^{N-1}\|z^{t+1} - z^t\|^2.
  \end{equation}
  To derive an upper bound for $\sum_{t=0}^{N-1}\|\A x^{t+1}_e\|^2$, observe from $\S := 2\Sigma_f - (\beta + \mu) \A^*\A\succeq \delta I$
  and \eqref{usefulbd} that
  \begin{equation*}
  \|x^{t+1}_e\|^2 \le \frac{1}{\delta}\|x^{t+1}_e\|^2_{\S} \le  \frac{1}{\delta}\left(\frac{1}{\gamma\beta}\|z_e^t\|^2 + \|y_e^t\|_\T^2\right) - \frac{1}{\delta}\left(\frac{1}{\gamma\beta}\|z_e^{t+1}\|^2 + \|y_e^{t+1}\|_\T^2\right);
  \end{equation*}
  and thus
  \begin{equation}\label{bdAx}
    \sum_{t=0}^{N-1}\|\A x^{t+1}_e\|^2 \le \|\A\|^2 \sum_{t=0}^{N-1}\|x^{t+1}_e\|^2 \le \frac{\|\A\|^2}{\delta}\left(\frac{1}{\gamma\beta}\|z_e^0\|^2 + \|y_e^0\|_\T^2 - \frac{1}{\gamma\beta}\|z_e^N\|^2 - \|y_e^N\|_\T^2\right)
  \end{equation}
  On the other hand, using the definition of $\sigma$ and choosing $\alpha = \beta$ in \eqref{usefulbd}, we obtain that
  \begin{equation*}
    \|z^{t+1} - z^t\|^2 \le \frac{\gamma^2\beta}{\sigma}\left(\frac{1}{\gamma\beta}\|z_e^t\|^2 + \|y_e^t\|_\T^2\right) - \frac{\gamma^2\beta}{\sigma}\left(\frac{1}{\gamma\beta}\|z_e^{t+1}\|^2 + \|y_e^{t+1}\|_\T^2\right),
  \end{equation*}
  which implies that
  \begin{equation}\label{bdz}
    \sum_{t=0}^{N-1}\|z^{t+1} - z^t\|^2 \le \frac{\gamma^2\beta}{\sigma}\left(\frac{1}{\gamma\beta}\|z_e^0\|^2 + \|y_e^0\|_\T^2\right) - \frac{\gamma^2\beta}{\sigma}\left(\frac{1}{\gamma\beta}\|z_e^{N}\|^2 + \|y_e^{N}\|_\T^2\right)
  \end{equation}
  Combining \eqref{bdAx} and \eqref{bdz} with \eqref{bdineq2} and the convexity of $f$ and $g$, the upper bound in \eqref{fvalbd} follows immediately.
  Furthermore, \eqref{constraintbd} follows immediately from \eqref{bdz}, the definition of $z^{t+1}$ and the convexity of norm function.
  Finally, the lower bound in \eqref{fvalbd} follows from \eqref{constraintbd} and the fact that
  \[
  f(x) + g(y) - \langle\bar z,\A x + \B y - c\rangle \ge f(\bar x) + g(\bar y),
  \]
  for all $x$, $y$, since $\bar z$ is a Lagrange multiplier for \eqref{genprimal}.
\end{proof}

The proximal AMA is closely related to the proximal ADMM, which has been studied extensively recently in the literature; see,
for example, \cite{Eckstein94,FPST12,HLHY02,XuW2011,YaZ09,ZBO09}. In essence, the latter algorithm is
also applicable to solve \eqref{genprimal} and is obtained by replacing the first subproblem in \eqref{scheme} with
\begin{equation}\label{subproblemADMM}
  x^{t+1}=\argmin_{x}\left\{f(x)-\langle z^t, \A x\rangle + \frac{\beta}{2}\|\A x + \B y^{t} - c\|^2 + \frac{1}{2}\|x - x^t\|^2_{\cal S}\right\},
\end{equation}
for some positive semidefinite linear map ${\cal S}$. In \cite[Theorem~B.1]{FPST12}, this algorithm is shown to converge
under conditions similar to those in Theorem~\ref{them_ama}, but has no restrictions in $\beta > 0$,
and allows an extra freedom in picking the stepsize in the $z$-update to be $\gamma\beta$
with $\gamma\in (0,\frac{\sqrt{5}+1}{2})$.

In the case when the proximal mapping of $\frac{1}{\beta}f$ is easy to compute,
the subproblem \eqref{subproblemADMM} for the proximal ADMM could be easy
to solve with a suitable choice of $\cal S$. Due to the extra freedom in choosing parameters, it is conceivable that
the proximal ADMM will work better than the proximal AMA.
As a consequence, for AMA, we do not consider the case when the proximal mapping of $\frac{1}{\beta}f$ is easy to compute
and thus did not add proximal term to the first subproblem of the proximal AMA.

\section{The proximal-proximal gradient algorithm}\label{sec4}

In this section, we apply the proximal AMA \eqref{scheme} developed in Section~\ref{sec:AMA} to \eqref{fencheldual}, or equivalently, \eqref{FBS}.
To this end, letting $\tau \ge \beta \|\M^*\M\|$ so that $\T:=\tau \I - \beta \M\M^*\succeq 0$, we replace
the second subproblem in \eqref{AMAsubproblem} by
\begin{equation}\label{choiceT}
y^{t+1} = \argmin_y \left\{ P^*(y) + \langle b,y\rangle - \langle{z^t}, \M^*y\rangle + \frac{\beta}{2}\|x^{t+1} + \M^*y\|^2 + \frac{1}{2}\|y - y^t\|^2_\T \right\},
\end{equation}
where $\|\cdot\|_\T$ is the semi-norm induced by the positive semidefinite linear map $\T$. Notice that the optimization
problem for this new subproblem has a unique solution due to the choice of $\T$. In particular, the iterates for this new algorithm are all well-defined, i.e., all the subproblems that arise have (unique) minimizers.

We next discuss how the subproblems of this new algorithm as applied to \eqref{FBS} (or, equivalently,
the proximal AMA as applied to \eqref{fencheldual}) can be solved efficiently.
From the first-order optimality condition, the first subproblem in \eqref{AMAsubproblem} amounts to finding $x^{t+1}$ so that
$z^t\in \partial h^*(x^{t+1})$, which can be obtained by setting $x^{t+1} = \nabla h(z^t)$ by \cite[Theorem~23.5]{Roc70}. Moreover, the second subproblem \eqref{choiceT} can
be reformulated as
\begin{align*}
  y^{t+1} & =
  \argmin_y \left\{ P^*(y) + \langle b,y\rangle - \langle{z^t}, \M^*y\rangle + \frac{\beta}{2}\|x^{t+1} + \M^*y\|^2 + \frac{1}{2}\|y - y^t\|^2_\T \right\}\\
          & = \argmin_y \left\{ P^*(y) + \langle b - \M z^t + \beta \M x^{t+1} + \beta \M\M^* y^t, y\rangle + \frac{\tau}{2}\|y - y^t\|^2 \right\}\\
          & = \argmin_y \left\{ P^*(y) + \frac{\tau}{2}\left\|y - \left(y^t - \frac{b - \M z^t + \beta \M x^{t+1} + \beta \M \M^* y^t}{\tau}\right)\right \|^2 \right\}\\
          & = {\rm prox}_{\tau^{-1}P^*}\left(\frac{\T y^t - b + \M z^t - \beta \M x^{t+1}}{\tau}\right),
\end{align*}
where $\T = \tau \I - \beta \M\M^*$.
Since the proximal mapping of $\tau P$ is easy to compute, the same is true of $\tau^{-1}P^*$, due to \eqref{proxmap}.
Precisely, we have
\begin{align}\label{alternate_yupdate}
  y^{t+1}& = \frac{\T y^t - b + \M z^t - \beta \M x^{t+1}}{\tau} -
 \argmin_u\left\{ P(\tau u) + \frac{\tau}{2}\left\|u - \frac{\T y^t - b + \M z^t - \beta \M x^{t+1}}{\tau}\right\|^2\right\}\nonumber \\
 & = \tau^{-1}\left(\T y^t - b + \M z^t - \beta \M x^{t+1} -
\argmin_v\left\{ P(v) + \frac{1}{2\tau}\left\|v - (\T y^t - b + \M z^t - \beta \M x^{t+1})\right\|^2\right\}\right)\nonumber\\
 & = \tau^{-1}\left(\T y^t - b + \M z^t - \beta \M x^{t+1} -
{\rm prox}_{\tau P}(\T y^t - b + \M z^t - \beta \M x^{t+1})\right).
\end{align}
Thus, the second subproblem \eqref{choiceT} can also be efficiently solved.

We note that the only difference between this new algorithm and the proximal gradient algorithm \eqref{FBSsubproblem} (which,
we recall, is equivalent to \eqref{AMAsubproblem})
lies in the second subproblem in \eqref{AMAsubproblem} (and also the choice of stepsizes $\beta>0$ and $\gamma > 0$),
where we simplified the subproblem by adding a proximal term. With our choice of $\T = \tau \I - \beta\M\M^*$, the subproblem
is reduced to a computation of the proximal mapping of $\frac{1}{\tau}P^*$. Hence,
we shall also refer to this new algorithm as the proximal-proximal gradient (PPG) algorithm.

We now summarize the PPG algorithm for solving \eqref{FBS} as follows:

\vspace{.1in}
\fbox{\parbox{5.6 in}{
\begin{description}
\item {\bf PPG algorithm for \eqref{FBS}}

\item[Step 0.] Input $(y^0, z^0)\in \Y\times \Z$, $\beta \in (0,\frac{2}{L})$, $\gamma\in (0,1 + \min\{\frac{1}{2},\frac{1}{\beta L}-\frac{1}{2}\})$,
$\tau \ge \beta \|\M^*\M\|$ and $\T = \tau \I - \beta \M\M^*$.

\item[Step 1.] Set
\begin{equation*}
\left\{
\begin{split}
&y^{t+1} = {\rm prox}_{\tau^{-1}P^*}\left(\frac{\T y^t - b + \M z^t - \beta \M\nabla h(z^t)}{\tau}\right),\\
& z^{t+1} = z^t - \gamma\beta (\nabla h(z^t) + \M^* y^{t+1}).
\end{split}
\right.
\end{equation*}

\item[Step 2.] If a termination criterion is not met, go to Step 1.
\end{description}
}}
\vspace{.1in}

We have the following convergence result concerning this new algorithm, which is a consequence of Theorem~\ref{them_ama}.
\begin{theorem}\label{them_ppg}
  Let $\{(y^t,z^t)\}$ be generated from the PPG algorithm for solving \eqref{FBS}
  and define $x^{t+1} = \nabla h(z^t)$ for all $t\ge 0$. Then
$\{(x^t, y^t)\}$ converges to an optimal solution to \eqref{fencheldual} and $\{z^t\}$ converges
to an optimal solution to \eqref{FBS}.
\end{theorem}
\begin{proof}
  Recall that $\{(x^t,y^t,z^t)\}$ thus generated is the same as the sequence obtained from the proximal AMA
  for solving \eqref{fencheldual}, using the same initial points, $\beta$, $\gamma$ and $\T$.

  We first show that {\bf A1} holds for \eqref{fencheldual}.
  To this end, let $\bar z$ be a solution of \eqref{FBS}. Since $h$ is continuous and differentiable,
  from the first-order optimality condition, \eqref{CQ} and the subdifferential calculus rules \cite[Theorem~23.8]{Roc70}, \cite[Theorem~23.9]{Roc70}, we have
  \[
  0 \in \nabla h(\bar z) + \M^*\partial P(\M\bar z - b).
  \]
  Hence, there exists $\bar y \in \partial P(\M\bar z - b)$ and $\bar x := \nabla h(\bar z)$ so that $\bar x + \M^* \bar y = 0$. Furthermore,
  notice from \cite[Theorem~23.5]{Roc70} that
  \[
  \bar y \in \partial P(\M\bar z - b) \Rightarrow \M\bar z - b\in \partial P^*(\bar y) \ \ {\rm and}\ \ \bar x = \nabla h(\bar z) \Rightarrow \bar z \in \partial h^*(\bar x).
  \]
  We thus conclude that \eqref{genopt} is satisfied with $h^*(x)$ in place of $f(x)$, $P^*(y) + \langle b,y\rangle$ in place of $g(y)$,
  the identity map $\I$ in place of $\A$, $\M^*$
  in place of $\B$ and $0$ in place of $c$.

  Next, recall from \cite[Theorem~2.1.5]{Nes03} that for any $z_1$, $z_2\in \Z$, we have
  \begin{equation*}
    \langle \nabla h(z_1) - \nabla h(z_2),z_1 - z_2 \rangle \ge \frac{1}{L}\|z_1-z_2\|^2,
  \end{equation*}
  since $\nabla h$ is Lipschitz continuous with modulus bounded by $L$.
  Hence $h^*$ satisfies \eqref{strongconvexf} in place of $f$ with $\Sigma_f = \frac{1}{L}\I$.

  The conclusion of the theorem now follows immediately from a direct application of Theorem~\ref{them_ama}, with $\Sigma_f = \frac{1}{L}\I$
  and $\Sigma_g = 0$.
\end{proof}

\begin{remark}
Unlike the proximal gradient algorithm \eqref{FBSsubproblem},
the sequence $\{\M z^t - b\}$ generated using $\{z^t\}$ from the above algorithm does not necessarily lie
in ${\rm dom}\,P$. This could be a disadvantage since then the objective function $F$ of \eqref{FBS} is not necessarily defined at each $z^t$.
However, when $\gamma = 1$, notice from \eqref{choiceT} that
\begin{equation}\label{domP}
\begin{split}
  &0 \in  \partial P^*(y^{t+1}) + \underbrace{b - \M z^t + \beta \M(x^{t+1} + \M^*y^{t+1})}_{b - \M z^{t+1}} + \T(y^{t+1} - y^t) \\
  \Rightarrow \ \ & y^{t+1} \in \partial P(\M z^{t+1} - b - \T(y^{t+1} - y^t)).
\end{split}
\end{equation}
Hence, in the case when $\gamma=1$ and $\M$ is surjective so that $\T = \M\cR$ for some linear map $\cR$, one can define
\[
\tilde z^{t+1} := z^{t+1} - \cR(y^{t+1} - y^t).
\]
Then we see immediately from \eqref{domP} that $\{\M\tilde z^t-b\}\subseteq {\rm dom}\,P$. Furthermore, since $\{y^t\}$ is convergent by Theorem~\ref{them_ppg},
it follows that $\lim_{t\rightarrow \infty} \tilde z^t = \lim_{t\rightarrow \infty} z^t$. Thus, the sequence $\{\tilde z^t\}$ is a feasible minimizing sequence for \eqref{FBS}.
\end{remark}

The PPG algorithm takes a particularly simple form when $h(z) = \frac{1}{2}\|z - \bar z\|^2$
for some $\bar z\in \Z$, as is illustrated in the next example.

\begin{example}
  Fix any $\bar z \in Z$ and consider the following optimization problem:
  \begin{equation}\label{projectK}
    \begin{array}{rl}
      \min\limits_z & \frac{1}{2}\|z - \bar z\|^2 + P(\M z - b),
    \end{array}
  \end{equation}
  which is a special case of \eqref{FBS} with $h(z) = \frac{1}{2}\|z - \bar z\|^2$.
  Since $\nabla h(z) = z - \bar z$, it is easy to see that one can take
  $L = 1$. Now, taking $\beta = \gamma = 1$, $\tau \ge \|\M^*\M\|$ and $\T = \tau \I- \M\M^*$ in the PPG algorithm, it is routine to show that
  \begin{equation}\label{projsub}
  \left\{
  \begin{split}
    &y^{t+1} ={\rm prox}_{\tau^{-1}P^*}\left(\frac{(\tau \I - \M\M^*)y^t + \M \bar z - b}{\tau}\right),\\
    &z^{t+1} = \bar z - \M^* y^{t+1};
  \end{split}
  \right.
  \end{equation}
  notice that the above updating rules are
  independent of $z^t$. In particular, one can completely ignore the $z$-update in the course of the algorithm.
  Furthermore, if $\M$ is surjective with $\T = \M\cR$ for some linear map $\cR$, one
  can obtain a feasible sequence $\{\tilde z^t\}$ converging to the optimal solution of \eqref{projectK} as
  \begin{equation*}
    \tilde z^{t+1} = \bar z - \M^*y^{t+1} - \cR(y^{t+1} - y^t).
  \end{equation*}

  On the other hand, one can show that the $y$-update in \eqref{projsub} is exactly
  the updating rule obtained by applying the proximal gradient algorithm \eqref{FBSsubproblem}
  to the dual of \eqref{projectK}. Indeed, according to \eqref{fencheldual}, the dual of \eqref{projectK} is given by
  \begin{equation}\label{projectKdual}
    \begin{array}{rl}
      \min\limits_{y} & \underbrace{\frac{1}{2}\|\M^* y\|^2 - \langle \M\bar z - b,y\rangle}_{h_2(y)} + P^*(y).
    \end{array}
  \end{equation}
  Notice that $\tau$ is an upper bound of the Lipschitz continuity modulus of the gradient of the smooth part $h_2(y)$ of the objective function in \eqref{projectKdual}. Thus, the updating rule \eqref{FBSsubproblem} applied to \eqref{projectKdual} yields the iterates
  \begin{equation*}
    y^{t+1} = {\rm prox}_{\tau^{-1}P^*}\left( y^t - \frac{1}{\tau}(\M\M^*y^t - \M\bar z + b)\right),
  \end{equation*}
  which is just the $y$-update in \eqref{projsub}.
\end{example}
%

We illustrate the PPG algorithm when $\M^*\M = \I$ or $\M\M^* = \I$ in the next example.
\begin{example}
Suppose that $\M^*\M = \I$ so that $\|\M^*\M\| = 1$. One can then take $\tau = \beta$ and $\gamma = 1$ in
the PPG algorithm. Using the definition of $z^{t+1}$, \eqref{alternate_yupdate}
and introducing the auxiliary iterates $u^t$ and $w^t$,
one can easily verify that the iterates of the algorithm can be represented as
\begin{equation*}
\left\{
\begin{split}
&w^{t+1} = \T y^t - b + \M z^t - \beta \M\nabla h(z^t),\ \ u^{t+1} = {\rm prox}_{\tau P}(w^{t+1}),\\
& z^{t+1} = \M^*\left(u^{t+1} + b\right),\ \ y^{t+1} = \frac{1}{\tau}(w^{t+1} - u^{t+1}).
\end{split}
\right.
\end{equation*}
On the other hand, if $\M\M^* = \I$, then setting $\tau = \beta$ will give $\T = 0$. Hence, no proximal term is added
to the alternating minimization algorithm when applied to the dual problem \eqref{fencheldual}, and
the PPG algorithm (with $\tau = \beta = \frac{1}{L}$ and $\gamma = 1$) for \eqref{FBS} reduces to the proximal gradient algorithm \eqref{FBSsubproblem}.
\end{example}

We consider the sum of several functions with simple proximal mappings in the next example.
\begin{example}
  Consider the following optimization problem:
  \begin{equation}\label{sumproblem}
    \begin{array}{rl}
      \min\limits_z & \displaystyle h(z) + \sum_{i=1}^mP_i(z),
    \end{array}
  \end{equation}
  where the proximal mappings for $\tau P_i$, $\tau > 0$ and $i=1,...,m$ are all easy to compute. The problem \eqref{sumproblem}
  is readily written in the form of \eqref{FBS} and hence the PPG algorithm is applicable. More precisely, define
  $\M z = (z,z,...,z) \in \Z^m$ and $P(u_1,...,u_m) = \sum_{i=1}^m P(u_i)$. Then we have $\sum_{i=1}^mP_i(z) = P(\M z)$.
  Moreover, the proximal mapping of $\tau P$, $\tau > 0$, is easy to compute.
\end{example}

Thanks to Theorem~\ref{them_complexity}, one can obtain various complexity results for the PPG algorithm
by using different choices of parameters. Here is one such result.

\begin{corollary}\label{cor_ppgcomplexity}
  Let $\{(y^t,z^t)\}$ be generated from the PPG algorithm for solving \eqref{FBS} with $\beta = \frac{1}{L}$,
  $\tau = \beta \|\M^*\M\|$ and $\gamma = 1$.
  Define $x^{t+1} = \nabla h(z^t)$ for all $t\ge 0$ and
  \begin{equation*}
    (\bar x^N,\bar y^N) = \frac{1}{N}\sum_{t=1}^N (x^t,y^t).
  \end{equation*}
  Then it holds that
  \begin{equation*}
  \begin{split}
    &-2\|\bar z\|\sqrt{\frac{L}{N}}
    \sqrt{L\|z^0 - \bar z\|^2 + \|y^0 - \bar y\|_{\tau \I - \beta\M\M^*}^2} \le h^*(\bar x^N) + P^*(\bar y^N) - h^*(\bar x) - P^*(\bar y) \\
    &\le
    \frac{1}{2N}\left(L\|z^0\|^2 + \|y^0 - \bar y\|_{\tau \I - \beta\M\M^*}^2\right) + \frac{1}{N}\left(L\|z^0 - \bar z\|^2 + \|y^0 - \bar y\|_{\tau \I - \beta\M\M^*}^2\right).
  \end{split}
  \end{equation*}
  and
  \begin{equation*}
    \|\A \bar x^N + \B \bar y^N - c\| \le 2\sqrt{\frac{L}{N}}
    \sqrt{L\|z^0 - \bar z\|^2 + \|y^0 - \bar y\|_{\tau \I - \beta\M\M^*}^2},
  \end{equation*}
  where $(\bar x,\bar y,\bar z)$ is the limit of $\{(x^t,y^t,z^t)\}$ guaranteed by Theorem~\ref{them_ppg}.
\end{corollary}
\begin{proof}
  Recall that the PPG algorithm in this case is just the proximal AMA applied to \eqref{fencheldual} with
  $\beta = \frac{1}{L}$ and $\gamma = 1$. Due to the choice of $\gamma$ and $\beta$, we see that
  $\delta$ can be chosen to be $\frac{1}{2L}$ (and thus $\mu=\frac{1}{2L}$, where $\mu$ is defined as in Theorem~\ref{them_complexity}) and $\sigma$
  can be chosen to be $\min\{\frac{1}{2},\frac{\mu L}{2}\} = \frac{1}{4}$ in Theorem~\ref{them_complexity}. The conclusion
  now immediately follows from Theorem~\ref{them_complexity}.
\end{proof}

\paragraph{Relationship with inexact proximal gradient algorithms.}
The PPG algorithm can be interpreted as a variant of an {\em inexact} proximal gradient algorithm.
Recall that an inexact proximal gradient algorithm is an algorithm in which the subproblem \eqref{FBSsubproblem} is only solved
inexactly in each iteration, and convergence is guaranteed if each subproblem \eqref{FBSsubproblem} is solved up to a certain accuracy;
see, for example, \cite{Patriksson98,SchRouBach11}.

To describe such an interpretation, we start by presenting one approach to solve the subproblem \eqref{FBSsubproblem} via an iterative method.
In this approach, we first notice that \eqref{FBSsubproblem} can be equivalently written as
\begin{align}
   &\min_z P(\M z - b) + \frac{L}{2}\left\|z - \left(z^t -\frac{1}{L}\nabla h(z^t)\right)\right\|^2\label{FBSsubproblemdualmin}\\
   = &\min_z \max_y \left\{\frac{L}{2}\left\|z - \left(z^t -\frac{1}{L}\nabla h(z^t)\right)\right\|^2 + \langle y,\M z - b\rangle - P^*(y)\right\}\nonumber\\
   \ge &\max_y \min_z \left\{\frac{L}{2}\left\|z - \left(z^t -\frac{1}{L}\nabla h(z^t)\right)\right\|^2 + \langle y,\M z - b\rangle - P^*(y)\right\}\nonumber\\
   = &\max_y \min_z -\frac{1}{2L}\|\M^* y\|^2 + \langle \M \hat z^t - b,y\rangle + \frac{L}{2}\left\|z - \hat z^t + \frac{1}{L}\M^* y\right\|^2 - P^*(y)\label{FBSsubproblemattain}\\
   = &\max_y \underbrace{-\frac{1}{2L}\|\M^* y\|^2 + \langle \M \hat z^t - b,y\rangle}_{-h_3(y)} - P^*(y)\label{FBSsubproblemdualmax},
\end{align}
where $\hat z^t := z^t -\frac{1}{L}\nabla h(z^t)$,
the first equality follows from the definition of convex conjugate of $P$ and \eqref{biconjugate}.
Furthermore, from \eqref{CQ} and \cite[Theorem~31.2]{Roc70},
we see that \eqref{FBSsubproblemdualmin} indeed equals \eqref{FBSsubproblemdualmax}. This implies that equality holds throughout in
the above relations. Thus, in order to solve for \eqref{FBSsubproblemdualmin}, one can equivalently solve for \eqref{FBSsubproblemdualmax}.
Moreover, it follows from \eqref{FBSsubproblemattain} that one can recover the optimal solution of \eqref{FBSsubproblemdualmin}
from an optimal solution $\tilde y^{t+1}$ of \eqref{FBSsubproblemdualmax} by setting
\begin{equation}\label{idealzupdate}
z^{t+1} = z^t - \frac{1}{L}(\nabla h(z^t) + \M^* \tilde y^{t+1}).
\end{equation}

We thus discuss an algorithm for solving \eqref{FBSsubproblemdualmax} instead.
Recall that the proximal mapping of $\tau^{-1}P^*$ is easy to compute for $\tau > 0$, assuming that of $\tau P$, $\tau > 0$,
is easy to compute; see \eqref{proxmap} and the discussions leading to \eqref{alternate_yupdate}. Hence, the maximization problem in \eqref{FBSsubproblemdualmax}
can be efficiently solved using the proximal gradient algorithm. In an inexact proximal gradient algorithm based on this scheme,
one solves \eqref{FBSsubproblemdualmax} sufficiently accurately by the proximal gradient algorithm and updates $z^{t+1}$ according to \eqref{idealzupdate},
using the approximate solution in place of $\tilde y^{t+1}$.

However, instead of getting a sufficiently accurate solution for \eqref{FBSsubproblemdualmax}, suppose we go to the other extreme:
apply {\em one step} of the proximal gradient algorithm to \eqref{FBSsubproblemdualmax}, initialized at $y^t$. Then, the resulting point
$y^{t+1}$ is given by
\begin{equation}\label{approxyupdate}
  y^{t+1} = {\rm prox}_{\tau^{-1}P^*}\left(y^t - \frac{\beta \M\M^* y^t - \M z^t + \beta \M\nabla h(z^t) + b}{\tau}\right)
\end{equation}
where $\beta = \frac{1}{L}$ and $\tau \ge \beta\|\M^*\M\|$, the Lipschitz continuity modulus of the gradient of $h_3$ in \eqref{FBSsubproblemdualmax}. Notice that the PPG algorithm with $\gamma = 1$
just consists of \eqref{approxyupdate} and \eqref{idealzupdate}, with $\tilde y^{t+1}$ replaced by $y^{t+1}$. In this sense,
the PPG algorithm can be viewed as an inexact proximal gradient algorithm with an (very) inaccurately solved subproblem, where we only take
{\em one step} of the proximal gradient algorithm to solve the dual of the subproblem.
This relationship between the PPG algorithm and the proximal gradient algorithm
resembles the relationship between the linearized Bregman iterative algorithm
and the Bregman iterative algorithm (see, for example, \cite[Section~5.3]{YODG08}), and that between
proximal ADMM and ADMM (see, for example, \cite[Section~2]{YaZ09}).

On the other hand,
this approach is fundamentally different from those considered in \cite{SchRouBach11}
which require the subproblem to be solved up to a certain accuracy. Hence, their convergence results do not directly apply to our algorithm.
The result in \cite[Section~4.2]{Patriksson98} is closer in spirit to our approach.
However, unlike there, in our algorithm, $z^{t+1} - z^t$ is
not a descent direction for $F(z)$ in general and no line search is needed in our algorithm.
Hence, their convergence results also do not directly apply to our algorithm.

\paragraph{Relationship with an algorithm proposed in \cite{Condat13}.}
The algorithms proposed in \cite{Condat13} apply to a more general objective function and allow inexact computation of gradients
and proximal mappings; see \cite{Vu13} for a further generalization of the algorithms into the settings of
finding zeroes for monotone operators. When restricted to solving our problem \eqref{FBS} using exact gradient and proximal mapping computations
and constant stepsize in extrapolation, one of their algorithms becomes quite similar to the PPG algorithm. More precisely,
that algorithm in this special case reduces to
\begin{equation}\label{Condatalg}
\left\{
\begin{split}
&z^{t+1} = z^t - \beta (\nabla h(z^t) + \M^* y^{t}),\\
&y^{t+1} = {\rm prox}_{\tau^{-1}P^*}\left(y^t - \frac{2\beta \M\M^* y^t - \M z^t + 2\beta \M\nabla h(z^t) + b}{\tau}\right),\\
&(z^{t+1},y^{t+1}) = \gamma(z^{t+1},y^{t+1}) + (1 - \gamma)(z^t,y^t).
\end{split}
\right.
\end{equation}
Although the formulae for updating the variables are similar to our algorithm,
there are quite a number of differences. First,
their order of updating the variables are different, and the extrapolation
is done for {\em both} variables. Moreover,
they use $\tau \I - 2\beta \M\M^*$ to scale $y^t$ instead of $\tau \I - \beta\M\M^*$.
Convergence result is given in \cite[Theorem~3.1]{Condat13}, where $\beta$ and $\tau$ have to be chosen
so that
\[
\frac{1}{\beta} - \frac{\|\M^*\M\|}{\tau}\ge \frac{L}{2}\ \ {\rm and}\ \
0<\gamma< 2 - \frac{L}{2}\left(\frac{1}{\beta} - \frac{\|\M^*\M\|}{\tau}\right)^{-1}.
\]
Thus, while $\beta$
can still be chosen from $(0,\frac{2}{L})$, their $\tau$ has to be chosen so that $\tau \ge \frac{\beta\|\M^*\M\|}{1 - \frac{\beta L}{2}}$,
which can be considerably larger than our choice of $\tau$ especially when $\beta$ is large.
Furthermore, unlike our algorithm where the maximum possible $\gamma$ depends {\em only} on $\beta$ and $L$,
it depends also on the choice of $\tau$ for their algorithm: the smaller the $\tau$, the closer the maximum possible $\gamma$ is to $1$,
meaning that the effect of extrapolation becomes smaller.
In view of
these differences, it does not seem that this algorithm proposed in \cite{Condat13} is equivalent to our algorithm.

\section{Numerical results}\label{sec5}

In this section, we perform numerical experiments to illustrate the performance of our algorithm. We consider two applications:
the system realization problem modeled via nuclear norm minimization \cite[Section~II.B]{LiV09} and the fused lasso logistic regression problem \cite[Section~5]{MaZhang13}.
Our codes are written in MATLAB. All numerical experiments are performed on an SGI XE340 system, with two 2.4 GHz
quad-core Intel E5620 Xeon 64-bit CPUs and 48 GB RAM, equipped with SUSE Linux Enterprise server 11
SP1 and MATLAB 7.14 (R2012a). All routines are timed using the tic-toc function in MATLAB.

\subsection{System realization problem}

As was formulated in \cite[Section~II.B]{LiV09} (see also \cite[Section~5]{FPST12}),
the system realization problem modeled as a nuclear norm minimization problem takes the following form:
\begin{equation}\label{sysreal}
  \min_{z\in \R^{m\times n(j+k-1)}} p_{\rm sys}(z) = \frac{1}{2}\|w\circ z - w\circ \hat z\|^2 + \lambda \|\cH(z)\|_*,
\end{equation}
where $z = \begin{pmatrix}
  z_0 & z_1 & \cdots & z_{j+k-2}
\end{pmatrix}$ with each $z_i\in \R^{m\times n}$, $\hat z$ is the given measurement, $w= \begin{pmatrix}
  w_0 & w_1 & \cdots & w_{j+k-2}
\end{pmatrix}\in \R^{m\times n(j+k-1)}$ is a zero-one matrix that are ones for the blocks with $i=0,...,k-1$ and is zero otherwise,
$\circ$ denotes the Hadamard (entry-wise) product, $\|\cdot\|$ denotes the Fr\"{o}benius norm (the norm induced by the trace inner product on $\R^{m\times n}$),
$\lambda > 0$ is the regularization parameter, $\|\cdot\|_*$ denotes the nuclear norm (the sum of all singular values), and
$\cH(z)$ is a block Hankel matrix defined as
\begin{equation*}
  \cH(z) = \begin{pmatrix}
    z_0&z_1&\cdots&z_{k-1}\\
    z_1&z_2&\cdots&z_{k}\\
    \vdots&\vdots&&\vdots\\
    z_{j-1}&z_{j}&\cdots&z_{j+k-2}
  \end{pmatrix}\in \R^{mj\times nk}.
\end{equation*}
It is easy to see that \eqref{sysreal} is in the form of \eqref{FBS} with
\begin{equation*}
\begin{split}
h(z) = \frac{1}{2}\|w\circ z - w\circ \hat z\|^2,\ P(u) = \lambda \|u\|_*,\
\M = \cH,\ b = 0.
\end{split}
\end{equation*}
From these we see that the condition \eqref{CQ} is trivially satisfied. Moreover,
as mentioned in the introduction, the proximal mapping of $\tau P$, $\tau > 0$, is easy to compute.
Furthermore, since $z\mapsto \|\cH(z)\|_*$ is coercive,
the set of optimal solutions of \eqref{sysreal} is nonempty. Hence, all three assumptions on \eqref{FBS} are satisfied for the specific problem \eqref{sysreal}.
For ease of reference, we also write down the dual of \eqref{sysreal} as follows:
\begin{equation}\label{sysrealdual}
\begin{array}{rl}
  \min\limits_{\nu,y} & d_{\rm sys}(\nu):= \frac{1}{2}\|\nu\|^2 + \langle w\circ \hat z,\nu\rangle \\
  {\rm s.t.} & w\circ \nu + \cH^*(y) = 0,\ \|y\|_\infty\le \lambda,
\end{array}
\end{equation}
where $\|y\|_\infty$ denotes the operator norm of $y\in \R^{mj\times nk}$, where $\nu$ is related to the $x$ in \eqref{fencheldual} via $x = w\circ \nu$.

In this subsection, we will perform numerical experiments to compare our PPG algorithm with the MFBS method \eqref{extragradient}
and the primal ADMM2 in \cite[Section~5.1]{FPST12} for solving \eqref{sysreal}. To apply our algorithm,
we have to determine a bound $L$ for the Lipschitz continuity modulus of $\nabla h$,
and obtain a bound of $\|\M^*\M\|$. It is easy to see one can take $L=1$.
Furthermore, it follows from \cite[Section~2]{FPST12} that $\|\cH^*\cH\| \le \min\{j,k\}$.
On the other hand, to apply the MFBS method \eqref{extragradient}, we need a bound $L_{\M}$ of the Lipschitz continuity modulus
of $G(z,y)$. By \cite[Section~6.2]{MonteiroSvaiter11},
$L_\M$ can be chosen to be $\frac{1}{2}(1 + \sqrt{1 + 4\min\{j,k\}})$. Finally, we refer the readers to \cite[Section~5.1]{FPST12} for
details about the primal ADMM2.

We initialize all three algorithms at the origin, i.e., $(y^0,z^0)= (0,0)$. We terminate the algorithms
when
\begin{equation}\label{termsys}
  \max\left\{\frac{|\min_{s\in \Gamma}p_{\rm sys}(z^s) + d_{\rm sys}(-w\circ \cH^*({\rm proj}_{\Omega}(y^t)))|}{\max\{\min_{s\in \Gamma}p_{\rm sys}(z^s),1\}},
  \frac{5\|\cH^*({\rm proj}_{\Omega}(y^t)) - w\circ \cH^*({\rm proj}_{\Omega}(y^t))\|}{\max\{\|\cH^*({\rm proj}_{\Omega}(y^t))\|,1\}}\right\} < tol
\end{equation}
for some $tol > 0$, where ${\rm proj}_{\Omega}$ denotes the projection onto the set $\Omega := \{y:\; \|y\|_\infty \le \lambda\},$\footnote{For
the MFBS method, we used $(u^t,v^t)$ in place of $(z^t,y^t)$ in \eqref{termsys}.
On the other hand, the projection in \eqref{termsys} is only performed for the primal ADMM2.}
the set $\Gamma := \{s:\;1\le s\le t, \mbox{ $s$ is a multiple of $10$}\}$, and the above criterion \eqref{termsys} is checked every $10$ iterations.
In our test below, we take $tol = 1e-4$ for all three algorithms. Also,
for our PPG algorithm, we take $\tau = \beta\min\{j,k\}$, $\gamma = 1 + 0.95\min\{0.5,\frac{1}{\beta L}-0.5\}$,
with $\beta = \frac{1}{L} = 1$ for $\lambda = 0.05$ and $\beta = 0.05$ otherwise,\footnote{While it is intuitively obvious
that for a fixed $\beta$, one should take $\gamma$ as large as possible and $\tau$ as small as possible so as to maximize
the stepsizes for updating $z$ and $y$, respectively, the situation is not so clear for the choice of $\beta$.
This is because $\beta$ is directly proportional to the stepsize in the $z$-update but inverse proportional to the stepsize in the $y$-update.
Thus, we experimented with $\beta = \frac{0.05}{L}$, $\frac{1}{L}$ and $\frac{1.95}{L}$. Setting
$\beta = \frac{0.05}{L}$ seems to work best for $\lambda \ge 0.1$; while for $\lambda = 0.05$, setting $\beta = \frac{1}{L}$ works best.}
while for the MFBS method
\eqref{extragradient}, we take $L_{\M} = \frac{1}{2}(1 + \sqrt{1 + 4\min\{j,k\}})$ and $\sigma=0.95$. We use the same parameters as used in
\cite[Section~5.1]{FPST12} for the primal ADMM2.

We generate random instances as in \cite[Section~II(B)]{LiV09}.
We start by generating random matrices $A\in \R^{r\times r}$, $B\in \R^{r\times n}$ and $C\in \R^{n\times r}$ with i.i.d. standard Gaussian entries
and normalize them to have operator norm $1$. We then randomly generate a $v_0\sim N(0,I)$ and $e_t\sim N(0,I)$ for $t=0,...,T-1$
with i.i.d. standard Gaussian entries, and create an ``output" $\tilde u_t$, $t=0,...,T-1$, according to the state space model:
\begin{align*}
  v_{t+1}&=Av_t+Be_t,\\
  \tilde u_t&=Cv_t+e_t.
\end{align*}
Random noise is then added to the output $\tilde u$ to give
$
\hat u = \tilde u + \sigma \epsilon,
$
with $\epsilon$ having i.i.d. standard Gaussian entries.
Finally, we set, for each $i=0,...,k-1$,
\begin{equation*}
  \hat z_i = \frac{1}{T}\sum_{t=0}^{T-1-i}\hat u_{t+i}\hat u_{t}^T,
\end{equation*}
and $\hat z_i=0$ for $i\ge k$.

In the test below, we fix $T=1000$, $m = n = 10$, $r = 10$, $j = 21$ and $\sigma = 5\times 10^{-2}$.
For each $k=100$, $200$ and $300$, and $\lambda = 0.05$, $0.1$ and $0.5$, we generate $10$ random instances as described above. The
computational results, averaged over the $10$ instances, are reported in Table~\ref{t1}, where we report the number of iterations ({\bf iter}), CPU time in seconds ({\bf cpu}),
primal objective value \eqref{sysreal} at termination ({\bf pobj}), dual objective value \eqref{sysrealdual} at termination ({\bf dobj})
and the relative dual infeasibility ({\bf dfeas})
\[\frac{\|\cH^*({\rm proj}_{\Omega}(y^t)) - w\circ \cH^*({\rm proj}_{\Omega}(y^t))\|}{\max\{\|\cH^*({\rm proj}_{\Omega}(y^t))\|,1\}}.\]
We observe that the primal ADMM2 is usually the fastest, and our PPG algorithm is usually faster than the MFBS method.

\begin{table}[h!]
\caption{\small Results for PPG algorithm, primal ADMM2 and MFBS method on solving \eqref{sysreal}}
 \label{t1}
\hspace{-1.8 cm}
\begin{footnotesize}
\begin{tabular}{|c c||r r r|r r r|r r r|}
\hline
 &  & \multicolumn{3}{c|}{\bf PPG} &
\multicolumn{3}{c|}{\bf P. ADMM2}&
\multicolumn{3}{c|}{\bf MFBS}\\
\hline
$k$ & \multicolumn{1}{c||}{$\lambda$} & \multicolumn{1}{c}{\bf iter} &
\multicolumn{1}{c}{\bf cpu}& \multicolumn{1}{c|}{\bf pobj/dobj/dfeas}
& \multicolumn{1}{c}{\bf iter} &
\multicolumn{1}{c}{\bf cpu}& \multicolumn{1}{c|}{\bf pobj/dobj/dfeas}
& \multicolumn{1}{c}{\bf iter} &
\multicolumn{1}{c}{\bf cpu}& \multicolumn{1}{c|}{\bf pobj/dobj/dfeas}\\
%

\hline

100 & 0.05 & 123 & 7.4  & 6.073e+0/6.072e+0/4.3e-6 & 40 & 2.7 & 6.073e+0/6.073e+0/9.6e-6 & 108 & 7.3  & 6.073e+0/6.073e+0/1.5e-5   \\
100 & 0.10 & 82  & 4.7  & 7.419e+0/7.419e+0/1.7e-5 & 78 & 5.1 & 7.419e+0/7.419e+0/4.2e-7 & 299 & 19.4 & 7.419e+0/7.419e+0/1.9e-6  \\
100 & 0.50 & 58  & 3.8  & 1.180e+1/1.180e+1/6.3e-6 & 21 & 1.5 & 1.180e+1/1.180e+1/4.2e-8 & 97  & 6.9  & 1.180e+1/1.180e+1/4.8e-6    \\
200 & 0.05 & 41  & 4.8  & 1.014e+1/1.014e+1/1.1e-5 & 40 & 5.2 & 1.014e+1/1.014e+1/5.2e-6 & 191 & 24.3 & 1.014e+1/1.014e+1/1.9e-5  \\
200 & 0.10 & 100 & 11.7 & 1.288e+1/1.288e+1/1.7e-5 & 55 & 7.3 & 1.288e+1/1.288e+1/3.9e-6 & 177 & 22.8 & 1.288e+1/1.288e+1/4.9e-6  \\
200 & 0.50 & 51  & 5.9  & 1.756e+1/1.755e+1/5.5e-6 & 20 & 2.5 & 1.755e+1/1.755e+1/4.2e-9 & 93  & 12.1 & 1.756e+1/1.755e+1/2.7e-6   \\
300 & 0.05 & 30  & 5.0  & 1.259e+1/1.259e+1/1.4e-5 & 38 & 7.0 & 1.259e+1/1.259e+1/6.4e-6 & 224 & 43.2 & 1.259e+1/1.259e+1/1.9e-5  \\
300 & 0.10 & 156 & 28.7 & 1.768e+1/1.768e+1/1.8e-5 & 43 & 9.1 & 1.768e+1/1.767e+1/1.3e-5 & 95  & 20.1 & 1.768e+1/1.767e+1/1.2e-5   \\
300 & 0.50 & 53  & 8.8  & 2.253e+1/2.253e+1/3.9e-6 & 20 & 3.6 & 2.253e+1/2.253e+1/2.3e-9 & 111 & 20.9 & 2.253e+1/2.253e+1/2.0e-6  \\

\hline
\end{tabular}
\end{footnotesize}
\end{table}

\subsection{Fused lasso logistic regression problem}

As discussed in \cite[Section~5]{MaZhang13}, the fused lasso logistic regression problem can be presented as follows:
\begin{equation}\label{fusedlasso}
  \min_{z\in \R^n} p_{\rm flasso}(z):=\sum_{i=1}^m \log(1 + \exp(-b_i(a_i^T\breve{z} + z_n))) + \lambda_1 \sum_{i=1}^{n-1}|z_i| + \lambda_2\sum_{i=1}^{n-2}|z_{i+1} - z_i|,
\end{equation}
where $a_i\in \R^m$ are samples, $b_i\in \{-1,1\}$, $i=1,...,m$, $m < n$, $\lambda_1\ge0$ and $\lambda_2\ge0$ are regularization parameters, $z\in \R^n$, $\breve{z}\in \R^{n-1}$ is the vector that contains the first $n-1$ entries of $z$ and $z_i$ denotes the $i$th entry of $z$.
It is easy to see that \eqref{fusedlasso} is in the form of
\eqref{FBS} with
\begin{equation*}
\begin{split}
h(z) = \sum_{i=1}^m \log(1 + \exp(-b_i(a_i^T\breve{z} + z_n))),\ \ \ \ \ \ \ \ \ \ \ \ \ \ \
\\ P(u) = \lambda_1\sum_{i=1}^{n-1}|u_i| + \lambda_2\sum_{i=n}^{2n-3}|u_i|,\
\M = M:= \begin{pmatrix}
  I_{n-1} & 0\\ E_{n-2} & 0
\end{pmatrix},\ b = 0,
\end{split}
\end{equation*}
where $M$ is a $(2n-3)\times n$ matrix,
$I_{n-1}$ is the identity matrix of dimension $n-1$ and $E_{n-2}$ is the $(n-2)\times (n-1)$ matrix whose diagonal entries are $1$,
upper diagonal entries are $-1$, and is zero otherwise. To further simplify notations, we define a matrix $A$ whose $i$th row is given by
$\begin{pmatrix}
  -b_ia_i^T & -b_i
\end{pmatrix}$ and a function $l(v):= \sum_{i=1}^m \log(1 + \exp(v_i))$. Then $h(z) = l(Az)$.
Moreover, it is routine to show that the conjugate function of $l$ is given by
\begin{equation*}
  l^*(u) = \sum_{i=1}^m\left(u_i\log(u_i) + (1-u_i)\log(1-u_i)\right),
\end{equation*}
with domain $0\le u_i\le 1$ for all $i$. Hence, the dual problem \eqref{fencheldual} of \eqref{fusedlasso} is given by
\begin{equation}\label{fusedlassodual}
\begin{array}{rl}
  \min\limits_{\nu,y} & d_{\rm flasso}(\nu):=\sum_{i=1}^m \left(\nu_i\log(\nu_i) + (1-\nu_i)\log(1-\nu_i)\right)\\
  {\rm s.t.} & A^T \nu + M^T y = 0,\ \max\limits_{1\le i\le n-1}|y_i| \le \lambda_1,\ \max\limits_{n\le i\le 2n-3}|y_i| \le \lambda_2,
\end{array}
\end{equation}
where $A^T$ and $M^T$ are the transpose of the matrices $A$ and $M$, respectively, and $\nu$ is related to the $x$ in \eqref{fencheldual} via $x = A^T\nu$.
Since $m < n$, under a reasonable assumption that $A$ has full row rank, we have $\nu = (A^T)^\dagger x$, where $^\dagger$ denotes the pseudoinverse.
First, as mentioned in the introduction, the proximal mapping of $\tau P$, $\tau > 0$, is easy to compute.
Moreover, since
\eqref{CQ} is trivially satisfied for \eqref{fusedlasso}, there is a zero duality gap between the primal problem
\eqref{fusedlasso} and the dual problem \eqref{fusedlassodual}.
From the dual problem, it is not hard to observe
that if $b_i$ are not all the same and both $\lambda_1$ and $\lambda_2$ are positive,
then the generalized Slater condition is satisfied for \eqref{fusedlassodual} and
an optimal solution to \eqref{fusedlasso} exists by \cite[Corollary~28.2.2]{Roc70}.
Hence, all assumptions on \eqref{FBS} are satisfied for the specific problem \eqref{fusedlasso},
under the reasonable assumptions that $b_i$ are not all the same (meaning that there are samples from the two different classes)
and both $\lambda_1$ and $\lambda_2$ are positive.

In this subsection, we will perform numerical experiments to compare our PPG algorithm against the MFBS method \eqref{extragradient} on solving \eqref{fusedlasso}. To apply our algorithm, we have to upper-bound the Lipschitz continuity modulus of $\nabla h$
and $\|\M^*\M\|$. To this end, it is routine to show that $\nabla l$ is Lipschitz continuous with modulus bounded by $0.25$. It then follows immediately from $h(z) = l(Az)$
that $\nabla h$ is Lipschitz continuous with modulus bounded by $0.25\lambda_{\max}(A^TA)$,
the maximum eigenvalue of $A^TA$. Next,
notice that
\begin{equation*}
  M^TM = \begin{pmatrix}
    2&-1&0&\cdots&\cdots&0&0\\
    -1&3&-1&\cdots&\cdots&0&0\\
    0&-1&3&\cdots&\cdots&0&0\\
    \vdots&\ddots&\ddots&\ddots&\ddots&\ddots&\vdots\\
    0&\cdots&\cdots&-1&3&-1&0\\
    0&\cdots&\cdots&0&-1&2&0\\
    0&\cdots&\cdots&0&0&0&0
  \end{pmatrix}\preceq
  \begin{pmatrix}
    3&-1&0&\cdots&\cdots&0&0\\
    -1&3&-1&\cdots&\cdots&0&0\\
    0&-1&3&\cdots&\cdots&0&0\\
    \vdots&\ddots&\ddots&\ddots&\ddots&\ddots&\vdots\\
    0&\cdots&\cdots&-1&3&-1&0\\
    0&\cdots&\cdots&0&-1&3&0\\
    0&\cdots&\cdots&0&0&0&0
  \end{pmatrix}=: M_2,
\end{equation*}
while the maximum eigenvalue of $M_2$ is bounded above by $5$; see, for example, \cite[Page~525]{Higham02}.
Hence, we have $\|\M^*\M\|\le 5$. On the other hand, to apply the MFBS method \eqref{extragradient}, we need to bound the Lipschitz continuity modulus
of $G(z,y)$. Using \cite[Lemma~6.2]{MonteiroSvaiter11} and the bound $\|\M^*\M\|\le 5$, we see that $L_\M$ can be chosen to be
$\frac{1}{2}(0.25\lambda_{\max}(A^TA) + \sqrt{(0.25\lambda_{\max}(A^TA))^2 + 20})$.

For simplicity, we initialize both algorithms at the origin, i.e., $(y^0,z^0)= (0,0)$. We terminate the algorithms
when
\begin{equation}\label{termflasso}
  \max\left\{\frac{|\min_{s\in {\Gamma}}p_{\rm flasso}(z^s) + d_{\rm flasso}(\tilde \nu^t)|}{\max\{\min_{s\in {\Gamma}}p_{\rm flasso}(z^s),1\}},
  \frac{5\|A^T\tilde \nu^t + M^T y^t \|}{\max\{\|A^T\tilde \nu^t\|,\|M^T y^t \|,1\}}\right\} < tol
\end{equation}
for some $tol > 0$,\footnote{For
the MFBS method, we used $(u^t,v^t)$ in place of $(z^t,y^t)$ in \eqref{termflasso}, and set $x^{t+1}=\nabla h(u^t)$ for all $t$.} where
\begin{equation*}
  \tilde \nu^t = \begin{cases}
    -(A^T)^\dagger M^Ty^t& {\rm if\ } 0\le -(A^T)^\dagger M^Ty^t\le 1,\\
   (A^T)^\dagger x^t & {\rm otherwise},
  \end{cases}
\end{equation*}
the set $\Gamma := \{s:\;1\le s\le t, \mbox{ $s$ is a multiple of $500$}\}$, and the above criterion \eqref{termflasso} is checked every $500$ iterations.
In our test below, we take $tol = 1e-4$ for both algorithms. Moreover,
for our PPG algorithm, we take\footnote{We experimented with $\beta = \frac{0.05}{L}$, $\frac{1}{L}$ and $\frac{1.95}{L}$. Setting
$\beta = \frac{1.95}{L}$ seems to work best for all values of $\alpha$.} $\beta = \frac{1.95}{L} = \frac{7.8}{\lambda_{\max}(A^TA)}$, $\tau = \frac{39}{\lambda_{\max}(A^TA)}$ and $\gamma = 1 + 0.95\min\{0.5,\frac{1}{\beta L}-0.5\}$, while for the MFBS method
\eqref{extragradient}, we take $L_{\M} = \frac{1}{2}(0.25\lambda_{\max}(A^TA) + \sqrt{(0.25\lambda_{\max}(A^TA))^2 + 20})$ as discussed above, and $\sigma=0.95$.

We generate random instances for our test. We start by generating an $m\times (n-1)$ matrix $C$
with i.i.d. standard Gaussian entries.
We then normalize $C$ so that each column has norm $1$ as in \cite[Section~4.1]{CandesTao07}. Moreover, mimicking \cite[Section~6.2]{MaZhang13} and \cite[Section~3.1]{YeXie11},
we take
\begin{equation*}
  \hat x_j:= \begin{cases}
    20\xi_1 &{\rm if \ }j = 1,2, . . . ,20,\\
    30\xi_2 &{\rm if \ }j = 41,\\
    10\xi_3 &{\rm if \ }j = 71, . . . ,85,\\
    20\xi_4 &{\rm if \ }j = 121, . . . ,125,\\
    0 &{\rm otherwise},
  \end{cases}
\end{equation*}
where $\xi_i$, $i=1,..,4$, are random numbers following standard Gaussian distribution. We then set $b$ to be the sign
vector of $C\hat x + \xi_5e$, where $\xi_5$ is a random number in $[0,1]$ and $e$ is the vector of all ones as in \cite[Section~6.2]{MaZhang13}.
We next form a matrix $C'$ by multiplying each column of $C$ entrywise with $-b$. The $m\times n$ input matrix $A$
for \eqref{fusedlasso} is then formed as $\begin{pmatrix}C' & -b\end{pmatrix}$. Furthermore, we set $\lambda_1 = \alpha m$ and $\lambda_2 = 100\lambda_1$
for some $\alpha > 0$, mimicking the choice in \cite[Section~6.2]{MaZhang13}.

We fix $m = 250$. For each $n = 10000$, $20000$, $30000$ and $\alpha = 1e-4$, $3e-4$ and $5e-4$,
we generate $10$ instances as described above. The computational results, averaged over the $10$ instances,
are reported in Table~\ref{t41}, where we report the number of iterations ({\bf iter}), CPU time in seconds ({\bf cpu}),
primal objective value \eqref{fusedlasso} at termination ({\bf pobj}), dual objective value \eqref{fusedlassodual} at termination ({\bf dobj})
and the relative dual infeasibility
\[\frac{\|A^T\tilde \nu^t + M^T y^t \|}{\max\{\|A^T\tilde \nu^t\|,\|M^T y^t \|,1\}}\]
at termination ({\bf dfeas}). We only report the results for our algorithm since the MFBS method \eqref{extragradient} never terminates
within $20000$ iterations. From Table~\ref{t41}, we observe that our algorithm performs reasonably well, and tends to be slower when $\alpha$ is smaller.

\begin{table}[h!]
\caption{\small Results for PPG algorithm on solving \eqref{fusedlasso}}
 \label{t41}\centering
\begin{footnotesize}
\begin{tabular}{|c c||r r r|}
\hline  
$n$ & \multicolumn{1}{c}{$\alpha$} & \multicolumn{1}{c}{\bf iter} &
\multicolumn{1}{c}{\bf cpu} & \multicolumn{1}{c|}{\bf pobj/dobj/dfeas}\\
%

\hline

  10000 & 1e-4 &  6450 &  16.7 & 1.167e+2/1.167e+2/1.1e-5 \\
  10000 & 3e-4 &  2400 &   6.1 & 1.560e+2/1.560e+2/7.9e-6 \\
  10000 & 5e-4 &  1500 &   3.8 & 1.670e+2/1.670e+2/5.3e-6 \\
  20000 & 1e-4 &  5700 &  39.4 & 1.262e+2/1.262e+2/1.3e-5 \\
  20000 & 3e-4 &  2950 &  22.1 & 1.491e+2/1.491e+2/9.3e-6 \\
  20000 & 5e-4 &  1600 &  12.5 & 1.631e+2/1.631e+2/4.4e-6 \\
  30000 & 1e-4 &  8150 &  97.7 & 1.144e+2/1.145e+2/1.1e-5 \\
  30000 & 3e-4 &  2900 &  31.1 & 1.504e+2/1.504e+2/7.2e-6 \\
  30000 & 5e-4 &  1850 &  21.9 & 1.652e+2/1.652e+2/8.8e-6 \\

\hline
\end{tabular}
\end{footnotesize}
\end{table}

\section{Concluding remarks}\label{sec6}

We have proposed a new algorithm for solving \eqref{FBS} which admits easy subproblems
assuming the proximal mappings of $\tau P$, $\tau > 0$, are easy to compute. Our algorithm
reduces to the usual proximal gradient algorithm when the affine map is just the identity map, and
is equivalent to applying a proximal AMA to the dual of \eqref{FBS}. We established
global convergence and discussed iteration complexity. Our computational results on solving
nuclear norm regularized system realization problem and the fused lasso logistic regression problem
show that our algorithm works reasonably well on large-scale instances.
%

\vskip 20 true pt

{\bf Acknowledgements.} The author would like to thank Christopher Jordan-Squire for carefully proofreading
an early version of this manuscript, and Stephen Beck for pointing out the important references \cite{Condat13} and \cite{Vu13}.


\begin{thebibliography}{99}
  \bibitem{BeckTeboulle09}
  A. Beck and M. Teboulle.
  \newblock A fast iterative shrinkage-thresholding algorithm for linear inverse problems.
  \newblock {\em SIAM J. Imaging Sci.} 2, pp. 183--202 (2009).

  \bibitem{BeckerCandesGrant11}
  S. Becker, E. J. Cand\`{e}s and M. Grant.
  \newblock Templates for convex cone problems with applications to sparse signal recovery.
  \newblock {\em Math. Program. Comput.} 3, pp. 165--218 (2011).

  \bibitem{BoLe06}
  J. M. Borwein and A. S. Lewis.
  \newblock {\em Convex Analysis and Nonlinear Optimization}.
  \newblock Springer, 2nd edition (2006).

  \bibitem{CandesTao07}
  E. J. Cand\`{e}s and T. Tao.
  \newblock The Dantzig selector: statistical estimation when $p$ is much larger than $n$.
  \newblock {\em Ann. Stat.} 35, pp. 2313--2351 (2007).

  \bibitem{Condat13}
  L. Condat.
  \newblock A primal-dual splitting method for convex optimization involving Lipschitzian, proximable and linear composite terms.
  \newblock {\em J. Optimiz. Theory App.} 158, pp. 460--479 (2013).

  \bibitem{Eckstein94}
  J. Eckstein.
  \newblock Some saddle-function splitting methods for convex programming.
  \newblock {\em Optim. Method Softw.} 4, pp. 75--83 (1994).

   \bibitem{EckB92}
  J. Eckstein and D. P. Bertsekas.
  \newblock On the Douglas-Rachford splitting method and the proximal point algorithm for maximal monotone operators.
  \newblock {\em Math. Prog.} 55, pp. 293--318 (1992).

  \bibitem{ElMiRu07}
  M. Elad, P. Milanfar and R. Rubinstein.
  \newblock Analysis versus synthesis in signal priors
  \newblock {\em Inverse Probl.} 23, pp. 947--968 (2007).

  \bibitem{FGlowinski83}
  M. Fortin and R. Glowinski.
  \newblock On decomposition-coordination methods using an augmented Lagrangian.
  \newblock In M. Fortin and R. Glowinski, eds., {\em Augmented Lagrangion Methods: Applications to the Solution of Boundary Problems.}
  North-Holland, Amsterdam (1983).

  \bibitem{FPST12}
  M. Fazel, T. K. Pong, D. Sun and P. Tseng.
  \newblock Hankel matrix rank minimization with applications to system identification and realization.
  \newblock {\em SIAM J. Matrix Anal. A.} 34, pp. 946--977 (2013).

  \bibitem{FukuMine81}
  M. Fukushima and H. Mine.
  \newblock A generalized proximal point algorithm for certain non-convex minimization problems.
  \newblock {\em Int. J. Syst. Sci.} 12, pp. 989--1000 (1981).

  \bibitem{Gabay83}
  D. Gabay.
  \newblock Applications of the method of multipliers to variational inequalities.
  \newblock In M. Fortin and R. Glowinski, eds., {\em Augmented Lagrangion Methods: Applications to the Solution of Boundary Problems.} North-Holland, Amsterdam, 1983.

  \bibitem{GaM76}
  D. Gabay and B. Mercier.
  \newblock A dual algorithm for the solution of nonlinear variational problems via finite element approximations.
  \newblock {\em Comput. Math. Appl.} 2, pp. 17--40 (1976).

  \bibitem{GlM75}
  R. Glowinski and A. Marroco.
  \newblock  Sur l'approximation, par elements finis d'ordre un, et la resolution, par penalisation-dualit'e, d'une classe de problemes de Dirichlet non lineares.
  \newblock {\em Revue Francaise d'Automatique, Informatique et Recherche Op'erationelle.} 9 (R-2), pp. 41--76 (1975).

  \bibitem{HLHY02}
  B. He, L. Liao, D. Han and H. Yang.
  \newblock A new
  inexact alternating directions method for monotone variational
  inequalities.
  \newblock{\em Math. Program.} 92, pp. 103--118 (2002).

  \bibitem{Higham02}
  N. J. Higham.
  \newblock {\em Accuracy and Stability of Numerical Algorithms}
  \newblock SIAM, Philadelphia, 2nd edition (2002).


  \bibitem{LanLuMon09}
  G. Lan, Z. Lu and R. D. C. Monteiro.
  \newblock Primal-dual first-order methods with $O(1/\epsilon)$ iteration-complexity for cone programming.
  \newblock {\em Math. Program.} 126, pp. 1--29 (2011).

  \bibitem{LiV08}
  Z. Liu and L. Vandenberghe.
  \newblock Interior-point method for nuclear norm
  approximation with application to system identification.
  \newblock {\em SIAM. J. Matrix Anal. A.} 31, pp. 1235--1256 (2009).

  \bibitem{LiV09}
  Z. Liu and L. Vandenberghe.
  \newblock Semidefinite programming methods for
  system realization and identification.
  \newblock {\em Proc. 48th IEEE
  Conference on Decision and Control}, pp. 4676--4681 (2009).

  \bibitem{LHV12}
  Z. Liu, A. Hansson and L. Vandenberghe.
  \newblock Nuclear norm system idenification with missing inputs and outputs.
  \newblock {\em Syst. Control Lett.} 62, pp. 605--612 (2013).

  \bibitem{MGC10}
  S.~Ma, D.~Goldfarb and L.~Chen.
  \newblock Fixed point and Bregman iterative methods for matrix rank minimization.
  \newblock {\em Math. Program.} 128, pp. 321--353 (2011).

  \bibitem{MaZhang13}
  S. Ma and S. Zhang.
  \newblock An extragradient-based alternating direction method for convex minimization.
  \newblock Preprint, Jan 2013. Available at \verb+http://arxiv.org/abs/1301.6308+.


  \bibitem{MoF10}
  K. Mohan and M. Fazel.
  \newblock Reweighted nuclear norm minimization with application to system identification.
  \newblock {\em Proc. American Control Conference} (2010).

  \bibitem{MonteiroSvaiter11}
  R. D. C. Monteiro and B. F. Svaiter.
  \newblock Complexity of variants of Tseng's modified F-B splitting and Korpelevich's methods for hemi-variational inequalities with applications to saddle point and convex optimization problems.
  \newblock {\em SIAM J. Optim.} 21, pp. 1688--1720 (2011).


  \bibitem{Nes83}
  Y. Nesterov.
  \newblock A method for solving a convex programming problem with convergence
  rate {$O(1/k^2)$}.
  \newblock{\em Soviet Math. Dokl.} 27(2), pp. 372--376 (1983).

  \bibitem{Nes03}
  Y. Nesterov.
  \newblock{\em Introductory Lectures on Convex Optimization: A Basic Course.}
  \newblock Kluwer Academic Publishers (2003).

  \bibitem{Nes05}
  Y. Nesterov.
  \newblock Excessive gap technique in nonsmooth convex minimization.
  \newblock {\em SIAM J. Optim.} 16, pp. 235--249 (2005).

  \bibitem{Nes052}
  Y. Nesterov.
  \newblock Smooth minimization of non-smooth functions.
  \newblock {\em Math. Program.} 103, pp. 127--152 (2005).

  \bibitem{Patriksson98}
  M. Patriksson.
  \newblock Cost approximation: a unified framework of descent algorithms for nonlinear programs.
  \newblock {\em SIAM J. Optim.} 8, pp. 561--582 (1998).


  \bibitem {PTJY09}
  T.~K. Pong, P. Tseng, S. Ji  and J. Ye.
  \newblock Trace norm
  regularization: reformulations, algorithms, and multi-task learning.
  \newblock {\em SIAM J. Optim.} 20, pp. 3465--3489 (2010).

  \bibitem{Roc70}
  R. T. Rockafellar.
  \newblock {\em  Convex Analysis}.
  \newblock Princeton University Press, Princeton (1970).

  \bibitem{ToYu10}
  K.-C. Toh and S. Yun.
  \newblock An accelerated proximal gradient algorithm for nuclear norm regularized least squares problems.
  \newblock {\em Pac. J. Optim.} 6, pp. 615--640 (2010).

  \bibitem{SchRouBach11}
  M. W. Schmidt, N. Le Roux and F. Bach.
  \newblock Convergence rates of inexact proximal-gradient methods for convex optimization.
  \newblock {\em NIPS}, pp. 1458--1466 (2011).

  \bibitem{ToSuSu12}
  R. Tomioka, T. Suzuki and M. Sugiyama.
  \newblock Augmented Lagrangian methods for learning, selecting, and combining features.
  \newblock In S. Sra, S. Nowozin and S. J. Wright (Eds.), {\em Optimization for Machine Learning}. MIT Press (2012).

  \bibitem{Tse91}
  P. Tseng.
  \newblock Applications of a splitting algorithm to decomposition in convex programming and variational inequalities.
  \newblock {\em SIAM J. Control Optim.} 29, pp. 119--138 (1991).

  \bibitem{Tse00}
  P. Tseng.
  \newblock A modified forward-backward splitting method for maximal monotone mappings.
  \newblock {\em SIAM J. Control Optim.} 38, pp. 431--446 (2000).


  \bibitem{Tse10}
  P. Tseng.
  \newblock Approximation accuracy, gradient methods, and error bound for structured convex
  optimization.
  \newblock  {\em Math. Program.} 125, pp. 263--295 (2010).

  \bibitem{V12}
  L. Vandenberghe.
  \newblock Convex optimization techniques in system identification.
  \newblock {\em Proc. IFAC Symposium on System Identfication}, pp. 71--76 (2012).

  \bibitem{Vu13}
  B. V\~{u}.
  \newblock A splitting algorithm for dual monotone inclusions involving cocoercive operators.
  \newblock {\em Adv. Comput. Math.} 38, pp. 667--681 (2013).

  \bibitem{WangBan13}
  H. Wang and A. Banerjee.
  \newblock Online alternating direction method.
  \newblock Preprint, June 2013. Available at \verb+http://arxiv.org/abs/1306.3721+.
%

  \bibitem{XuW2011}
  M.~H.~Xu and T.~Wu.
  \newblock A class of linearized proximal alternating direction methods.
  \newblock {\em J. Optim. Theory Appl.} 151, pp. 321--337 (2011).

  \bibitem{YaZ09}
  J. Yang and Y. Zhang.
  \newblock Alternating direction
  algorithms for $\ell_1$-problems in compressive sensing.
  \newblock {\em SIAM J. Sci. Comput.} 33, pp. 250--278 (2011).

  \bibitem{YeXie11}
  G. Ye and X. Xie.
  \newblock Split Bregman method for large scale fused Lasso.
  \newblock {\em Comput. Stat. Data An.} 55, pp. 1552--1569 (2011).

  \bibitem{YODG08}
  W. Yin, S. Osher, J. Darbon and D. Goldfarb.
  \newblock Bregman iterative algorithms for compressed sensing and related problems.
  \newblock {\em SIAM J. Imaging Sci.} 1, pp. 143--168 (2008).

  \bibitem{ZBO09}
  X. Zhang, M. Burger and S. Osher.
  \newblock A unified
  primal-dual algorithm framework based on Bregman iteration.
  \newblock  {\em J. Sci. Comput.} 46, pp. 20--46 (2011).
\end{thebibliography}
\end{document}